\documentclass[10pt]{article}
\usepackage{color}
\usepackage{amssymb}
\usepackage{amsthm,array,amssymb,amscd,amsfonts,latexsym, url}
\usepackage{amsmath}
\usepackage[all]{xy}
\newtheorem{theo}{Theorem}[section]
\newtheorem{prop}[theo]{Proposition}

\newtheorem{lemm}[theo]{Lemma}
\newtheorem{coro}[theo]{Corollary}
\newtheorem{rema}[theo]{Remark}
\newtheorem{Defi}[theo]{Definition}
\newtheorem{ex}[theo]{Example}

\newtheorem{conj}[theo]{Conjecture}

\title{On the Lefschetz standard conjecture for Lagrangian covered  hyper-K\"{a}hler varieties}
\author{Claire Voisin\footnote{The author is supported by the ERC Synergy Grant HyperK (Grant agreement No. 854361).}}

\date{}

\newfont{\gothic}{eufb10}
\begin{document}
\maketitle
\setcounter{section}{-1}

\begin{abstract} We investigate the Lefschetz standard conjecture for degree $2$ cohomology of hyper-K\"ahler manifolds admitting a covering by Lagrangian subvarieties.  In the case of a Lagrangian fibration, we show that the Lefschetz standard conjecture is implied
by the SYZ conjecture  characterizing  classes of divisors associated with Lagrangian fibration. In dimension $4$, we consider  the more general case of a Lagrangian covered fourfold $X$, and prove   the Lefschetz standard conjecture in degree $2$, assuming $\rho(X)=1$ and $X$ is general in moduli. Finally we discuss various  links between Lefschetz cycles and  the study of the rational equivalence of points and Bloch-Beilinson type filtrations, giving a general interpretation of  a recent intriguing  result of Marian and Zhao.
 \end{abstract}
\section{Introduction}
The Lefschetz standard conjecture \cite{kleiman} for degree $k$ cohomology of a smooth projective variety of dimension
$N\geq k$  equipped with an ample line bundle $H_X$  asks whether there exists a
codimension $k$-cycle $\mathcal{Z}_{{\rm lef}}\in{\rm CH}^k(X\times X)$ such that, denoting $[\mathcal{Z}_{{\rm lef}}]\in H^{2k}(X\times X,\mathbb{Q})$ the cohomology class of $\mathcal{Z}_{{\rm lef}}$,
$$[\mathcal{Z}_{\rm lef}]^*: H^{2N-k}(X,\mathbb{Q})\rightarrow H^k(X,\mathbb{Q})$$
is the inverse of the Lefschetz isomorphism
$$h_X^{N-k}\cup :H^k(X,\mathbb{Q})\rightarrow H^{2N-k}(X,\mathbb{Q}),$$
where $h_X=c_1(H_X)$.
\begin{rema}\label{remaintrolefpaspol}{\rm  The Lefschetz standard conjecture is in fact  independent of the choice of polarization $h_X$ and it can be  formulated  without any reference to  this choice. Indeed, it is equivalent to the existence of a cycle $\mathcal{Z}_{\rm lef}$ such that  $[\mathcal{Z}_{\rm lef}]^*: H^{2N-k}(X,\mathbb{Q})\rightarrow H^k(X,\mathbb{Q})$ is an isomorphism. However, its implications rely on  the Hodge-Riemann-Lefschetz theory which uses  an explicit polarization.}
\end{rema}
\begin{rema}\label{remaintro}
{\rm When $k=2$, it suffices to show that $[\mathcal{Z}_{\rm lef}]^*: H^{N,N-2}(X)\rightarrow H^{2,0}(X)$
is the inverse of the Lefschetz isomorphism
$h_X^{N-2}\cup :H^{2,0}(X)\rightarrow H^{N,N-2}(X,\mathbb{Q})$. Indeed, the equality of morphisms
of Hodge structures $[\mathcal{Z}_{\rm lef}]^*\circ h_X^{N-2}\cup=Id$ is true on $H^2(X,\mathbb{Q})_{tr}$ once it is true on $H^{2,0}(X)$, and furthermore  we have $H^2(X,\mathbb{Q})=H^2(X,\mathbb{Q})_{tr}\oplus {\rm NS}(X)_{\mathbb{Q}}$ where the direct sum is orthogonal with respect to the nondegenerate Lefschetz intersection pairing $(\alpha,\beta)_{h_X}=\langle \alpha,h_X^{N-2}\beta\rangle_X$. Once we have a cycle $\mathcal{Z}$ satisfying  $[\mathcal{Z}]^*\circ h_X^{N-2}\cup=Id$   on $H^2(X,\mathbb{Q})_{tr}$, it is easy to construct  $\mathcal{Z}_{\rm lef}$  by adding to
$\mathcal{Z}$ a decomposable cycle in ${\rm NS}(X)\otimes {\rm NS}(X)$, so that $[\mathcal{Z}_{\rm lef}]^*\circ h_X^{N-2}\cup=Id$  holds   on the whole $H^2(X,\mathbb{Q})$.
}
\end{rema}

This conjecture is essential for the theory of motives as it allows to realize motivically the Lefschetz decomposition, which in turn implies  semisimplicity results for cohomological motives. The Lefschetz standard conjecture also implies the variational Hodge conjecture (see \cite{voisinhodge}). Except for hyper-K\"ahler manifolds that we will discuss in this paper,
the main class of varieties for which the  Lefschetz standard conjecture is known is the class of  abelian varieties by work of Lieberman \cite{lieberman}.
In degree $2$,  we will describe several equivalent formulations in Section \ref{seclef}, with emphasis on its relationship to the  Chow group of $0$-cycles modulo rational equivalence.

 If $X$ is a projective hyper-K\"ahler manifold of dimension $N=2n$, $H^{2,0}(X)$ is $1$-dimensional, generated by a $2$-form  $\sigma_X$. For such an $X$, the Lefschetz standard conjecture for degree $2$
 is equivalent to the existence of a
codimension $2$-cycle $Z\in{\rm CH}^2(X\times X)$ whose cohomology class $[Z]\in H^4(X\times X,\mathbb{Q})$ gives by K\"{u}nneth decomposition the inverse $q^{-1}\in {\rm Sym}^2H^2(X,\mathbb{Q})$ of  the Beauville-Bogomolov
form $q$ on $H^2(X,\mathbb{Q})$. This follows from the fact that the Beauville-Bogomolov form, restricted
to the transcendental cohomology, is a multiple of the Lefschetz intersection pairing
$q_{\rm lef}(\alpha,\beta)=\int_Xh_X^{2n-2}\alpha\beta$, as follows from the Beauville-Fujiki relations.
 The existence of such a cycle has been proved by Markman
\cite{markm} when $X$ is of $K3^{[n]}$-deformation type.
The Lefschetz standard conjecture  has been proved in all degrees  by Charles and Markman \cite{charlesmarkman}, following
\cite{markm} and  a programme presented in \cite{charles},
for projective hyper-K\"ahler manifolds of $K3^{[n]}$-deformation type. Note also that
the Lefschetz standard conjecture for degree $2$ is satisfied by  most of the explicitly known hyper-K\"ahler manifolds. For example,
it is true for the Fano variety of lines of a cubic fourfold \cite{bedo}, the LLSvS $8$-fold constructed from the family of degree $3$ rational curves in a smooth cubic fourfold \cite{llsvs}, the LSV $10$-fold constructed in \cite{lsv} as a compactification of the intermediate Jacobian fibration, and more generally the infinitely many families of  hyper-K\"ahler manifolds  constructed in \cite{bayeretal}. This follows, using Proposition \ref{procharles}, from the existence of a natural correspondence between these varieties and the corresponding cubic fourfold, whose motive is a direct summand in the motive of a surface. Of course, except in the case of the LSV variety, one can also apply the Charles-Markman theorem, since they are all of $K3^{[n]}$ deformation type. Similarly, the double EPW sextics constructed by O'Grady \cite{epw} satisfy the conjecture, either by using the correspondence to a Fano fourfold constructed by Iliev and Manivel \cite{ilma}, or by applying \cite{charlesmarkman}.

Our goal in this paper is to study the Lefschetz standard conjecture for degree $2$ cohomology of hyper-K\"ahler manifolds
which admit a Lagrangian fibration or covering. Our first result concerns the fibered case.
A central conjecture in the theory of hyper-K\"ahler manifolds is the following precise version of the SYZ conjecture.
\begin{conj}\label{conjlag} Let $X$ be a projective hyper-K\"ahler manifold and $l\in {\rm NS}(X)$  satisfying
the conditions

(1)  $q(l)=0$,

(ii) $l$ belongs to the boundary of the birational K\"ahler cone.

Then there exist a  hyper-K\"ahler manifold $\psi: X\dashrightarrow X'$ birational to $X$ and a  Lagrangian fibration $X'\rightarrow B$ on $X'$ such that $\psi^*l'=l$, where $l'$ is the Lagrangian class of $\pi'$.
\end{conj}
Our first result is the following.
\begin{theo}\label{theobasic} Let $X$ be a projective hyper-K\"ahler manifold. Assume that $X$ admits a Lagrangian fibration
$\pi:X\rightarrow B$ and that $\rho(X)=2$. Then if one of the following conditions

\begin{enumerate}
\item \label{condSYZ}  $X$ satisfies the SYZ conjecture \ref{conjlag};

\item \label{condtheta}  There is an effective divisor $\Theta$ on $X$ which satisfies $q(\theta)<0$;
\end{enumerate}

holds, $X$ satisfies the Lefschetz standard conjecture for degree $2$ cohomology.
\end{theo}
Here $\theta=[\Theta]$ and $q$ is the Beauville-Bogomolov quadratic form on $H^2(X,\mathbb{Q})$.
We now turn to the more general case of a Lagrangian covering, with the following definition.
\begin{Defi} \label{defigoodcov} A   covering of a hyper-K\"ahler manifold $X$ by Lagrangian varieties is a diagram
of morphisms
$$
\begin{matrix}
 & L&\stackrel{\Phi}{\longrightarrow}& X
 \\
&\pi\downarrow& &
\\
&B& &
\end{matrix}
$$
of smooth projective varieties, such that $\Phi$ is dominant and,  for a general point $b\in B$,
 $\Phi_{\mid L_b}$ is generically finite to its image which is a (singular) Lagrangian subvariety of $X$.
\end{Defi}
In practice, we can assume that the map $\Phi_{\mid L_b}$ is  birational to its image.
\begin{ex}\label{examplenonum}{\rm  If $X$ contains a smooth Lagrangian subvariety $L_0\subset X$ whose normal bundle $N_{L_0/X}$ is
generated by global sections, then $X$ admits a covering by Lagrangian varieties. This follows from the unobstructedness result of \cite{voisinlag} for deformations of Lagrangian submanifolds of hyper-K\"ahler manifolds. }
\end{ex}
 It is not known if such a covering exists for any projective hyper-K\"ahler manifold but it exists on
general members of some locally complete families  of hyper-K\"ahler manifolds with Picard number $1$ as the following example shows.
\begin{ex}\label{exfanolagcov}{\rm Let $Y\subset \mathbb{P}^5$ be a smooth cubic fourfold and $X=F_1(Y)$ be its Fano variety of lines. This is a hyper-K\"ahler fourfold by \cite{bedo}.
For any hyperplane $H\subset \mathbb{P}^5$, the hyperplane section $Y_H=Y\cap H$ is a cubic threefold, whose variety of lines is, if $H$ is general, or for any $H$ if $Y$ is general, a surface $\Sigma_H\subset X$ which is Lagrangian, as is observed in \cite{voisinlag}. The family of these surfaces gives a  Lagrangian covering of $X$.}
\end{ex}
 O'Grady conjectured it should exist in general. In contrast,  the Lagrangian fibered hyper-K\"ahler manifolds have $\rho\geq 2$, so a priori Theorem \ref{corodim4}
applies to a broader setting in dimension $4$.

\begin{theo}\label{corodim4} Let  $X$ be a projective hyper-K\"ahler fourfold with $\rho(X)=1$ admitting a covering by Lagrangian surfaces. Assume $X$ is very general in moduli.
Then $X$ satisfies the Lefschetz standard conjecture for degree $2$ cohomology.
\end{theo}
\begin{rema}\label{remaspec} {\rm  The Lefschetz cycle $\mathcal{Z}_{\rm lef}^2$ that exists at the fiber over  a very general point can be specialized to the fiber over any point in the moduli space of smooth deformations, proving in turn the Lefschetz standard conjecture  in degree $2$ for any smooth  member   $X_t$ of the moduli space of polarized deformations of $X$.}
\end{rema}

\begin{coro}\label{corodim4gen}  Let  $X$ be a projective hyper-K\"ahler fourfold with $\rho(X)=1$ containing a smooth Lagrangian surface
$L_0$ whose normal bundle is globally generated.  Then $X$ satisfies the Lefschetz standard conjecture for degree $2$ cohomology.
\end{coro}
\begin{proof}   By using the deformations of $L_0$ in $X$  as explained in Example \ref{examplenonum}, $X$ admits a covering by Lagrangian surfaces.
As $\rho(X)=1$, we know by \cite{voisinlag} that the existence of $L_0$  is satisfied  in a Zariski open set of
the moduli space of polarized deformations of $X$. Furthermore, as the normal bundle of a Lagrangian submanifold is isomorphic to its cotangent bundle, the dimension of its space of global sections remains constant under deformations  and thus,  the normal bundle remains globally generated
for the deformed pair $L_{0,t}\subset X_t$, at least for $t$ in a Zariski dense open set of the moduli space. For a very general deformation  $X_t$ of $X$,    Theorem \ref{corodim4} thus applies. We conclude the proof for the original $X$ using Remark \ref{remaspec}.
\end{proof}

\begin{rema}{\rm The assumptions of Corollary \ref{corodim4gen} are satisfied in Example \ref{exfanolagcov}.
}
\end{rema}

In Section \ref{secprooflagcov}, we will also establish  Propositions \ref{proassezgencriterelef} and
\ref{coroprovisoire} which provide two   steps towards proving  the Lefschetz standard conjecture in degree $2$ for Lagrangian covered
hyper-K\"ahler manifolds of any dimension, assuming the base $B$ of the Lagrangian covering has $H^{2,0}(B)=0$ (an assumption that we do not need in dimension $4$).

The second theme of this paper, developed in Section \ref{SEC2},  is the relation between  the Lefschetz standard conjecture in degree $2$ and
rational equivalence of points on smooth projective  varieties  whose algebra of holomorphic forms is generated in degree $\leq 2$.
In the section \ref{seclef} devoted to general facts about
Chow groups of $0$-cycles and Lefschetz conjecture, we will  define geometrically the level $F^3_{dR}{\rm CH}_0(X)\subset F^2{\rm CH}_0(X)$ for any smooth complex projective varieties, where $F^2{\rm CH}_0(X)$ is the group of zero-cycles of degree $0$ and Albanese equivalent to $0$, and we will
establish in Theorem  \ref{propeqlef} a precise relationship between the Lefschetz conjecture for degree $2$ cohomology
 and $F^3_{dR}{\rm CH}_0(X)$. Whether this subgroup coincides or not with other definitions $F^3_{BB}{\rm CH}_0(X)$ proposed in
 \cite{voisinanali} for the  $F^3$ level of the Bloch-Beilinson filtration  precisely depends on
 the Lefschetz conjecture in degree 2 and the Bloch conjecture for $0$-cycles on a surface.
 We propose and discuss  in Section \ref{secconjzero}
a conjecture  (see Conjecture \ref{conjsurpoints}) on the rational equivalence of points of a smooth projective hyper-K\"ahler manifold.
It states  that two points $x, \,y$ have the same class in ${\rm CH}_0(X)$ if and only they have the same class in
 ${\rm CH}_0(X)/F^3_{BB}{\rm CH}_0(X)$, where we refer to Section \ref{seclef} for the definition of $F^3_{BB}$.
 This conjecture
 is motivated by the recent work \cite{marian} of  A. Marian and  X.  Zhao, which is an interesting evidence for it. We discuss further evidence  and relate it to a conjectural polynomial formula expressing partially the diagonal of a  hyper-K\"ahler manifold as a polynomial
 in the Lefschetz cycle. As will be discussed there, the conjecture should hold more generally for smooth projective varieties whose algebra of holomorphic forms is generated in degree $\leq 2$,  but in the hyper-K\"ahler case, especially the Lagrangian fibered case, it seems more accessible, thanks to polynomial relations in the Chow ring that are expected to hold (see \cite{beauvoi}, \cite{shenvial}) and are known to hold in some cases, (see  \cite{voisinpamq}).
  We also describe  some  classes of varieties for which the conjecture holds. For example, using Beauville's formulas, we prove in Proposition \ref{propkummer} that  two points $x,\,y$ of a  Kummer variety $K=A/\pm Id$ of an abelian variety are rationally equivalent if and only if $x=y$ modulo $F^3_{BB}{\rm CH}_0(K)$, if and only if
 $x_{\leq 2}=y_{\leq 2} \in{\rm CH}_0(K)_{\leq 2}$, see \cite{beufourier} for the definition of the Beauville decomposition which is used to define ${\rm CH}_0(K)_{\leq 2}$).

We use the notation   ${\rm CH}^i(X)$ for the Chow groups with $\mathbb{Q}$-coefficients.

\vspace{0.5cm}

{\bf Thanks.} {\it I thank Giulia Sacc\`a for interesting discussions on the subject of this paper and the anonymous referee for  his/her careful reading and useful comments.}

\section{The Lefschetz standard conjecture for hyper-K\"{a}hler manifolds swept-out by Lagrangian varieties\label{secleftheosproof}}
This section is devoted to the proof of Theorems \ref{theobasic} and \ref{corodim4}.
We will use the following result  proved in \cite{charles}.
\begin{prop}\label{procharles} Let $X$ be a smooth projective variety. Then the Lefschetz standard conjecture holds
for degree $2$ cohomology of $X$ if and only if there exist a smooth projective surface $\Sigma$ and a
codimension $2$ cycle
$Z_\Sigma\in {\rm CH}^2(X\times \Sigma)$, such that
$$[Z_\Sigma]^*: H^{2,0}(\Sigma)\rightarrow H^{2,0}(X)$$
is surjective.
\end{prop}
If the Lefschetz standard conjecture holds for a very ample line bundle $H_X$, one can indeed take for $\Sigma$  a smooth surface which is a complete intersection
of hypersurfaces in $|H_X|$ and for $Z_\Sigma$ the cycle $\mathcal{Z}_{\rm lef}$ restricted to $X\times \Sigma$. The other direction is more tricky.

\subsection{Hyper-K\"{a}hler manifolds admitting a Lagrangian fibration\label{secprooftheolag}}
Let $X$ be a smooth projective complex manifold of dimension $2n$ with a closed  holomorphic $2$-form $\sigma_X$, and  $\pi: X\rightarrow B$ be a  fibration whose fibers are isotropic  with respect to $\sigma_X$. In this section, we will consider
  the case of a Lagrangian fibration,  where the $2$-form is  nondegenerate and ${\rm dim}\,B=n$. In particular, the general fiber of $\pi$ is an abelian variety.
There is a family of $n-1$-cycles on $X$ parameterized by $X$ and defined as follows: let $H$ be an effective  ample divisor   on $X$ with restriction $H_b\subset X_b$ on the general fiber $X_b$ of $\pi$. Let $S\subset X$ be a multisection of $\pi$, of degree $D$ over $B$.
Then for $x\in X_b$, translation by ${\rm alb}_{X_b}(Dx-S_b)$ acts on $X_b$, hence on the set of effective  divisors of $X_b$. We denote
by $H_{b,x}\subset X_b$ the effective divisor obtained by applying this action to $H_b$. This is an effective  $n-1$-cycle in $X$ which is determined by $x\in X_b$, and $x\times H_{b,x}$ is an effective  $n-1$-cycle in $X\times X$, which is well defined for a general point $x\in X$. The Zariski closure in $X\times X$ of the union  of these
$n-1$-cycles provides a self-correspondence $\Gamma_\pi\in {\rm CH}^{n+1}(X\times X)$ with action
$$[\Gamma_\pi]^*(\alpha):={\rm pr}_{1*}([\Gamma_\pi]\cup {\rm pr}_2^*\alpha)$$ on the cohomology of $X$.
 Let  $h:=[H]\in H^2(X,\mathbb{Z})$,  and let $\widetilde{B}$ be a desingularization of $B$. If $X$ is hyper-K\"ahler, let $l$ be the Lagrangian class associated with the fibration $\pi$, that is, $l$ comes from the positive generator of
  ${\rm NS}({B})$ (see \cite{matsubasic}). We prove the following proposition.
\begin{prop}\label{prodebaseactGamma} Assuming that $H^0(\widetilde{B},\Omega_{\widetilde{B}}^2)=0$, one has \begin{eqnarray}
\label{eqaction} [\Gamma_\pi]^*(h^{n-1}\sigma_X)=c \sigma_X
\end{eqnarray}
for some nonzero rational  constant $c$.
Furthermore, if $X$ is hyper-K\"ahler,
\begin{eqnarray}
\label{eqaction2} [\Gamma_\pi]^*(h^{2n-2}\sigma_X)=c_1 l^{n-1} \sigma_X
\end{eqnarray}
 for some nonzero rational  constant  $c_1$.
\end{prop}
\begin{proof} Let us   first show that (\ref{eqaction}) implies (\ref{eqaction2}). The correspondence $\Gamma_\pi$ is a relative correspondence, which means that it is contained in $X\times_B X$. It then follows from the projection formula that
\begin{eqnarray}
\label{eqactionutile0} [\Gamma_\pi]^*(l^k\alpha)=l^k[\Gamma_\pi]^*(\alpha)
\end{eqnarray}
 for any $\alpha\in H^*(X)$. We can write
\begin{eqnarray}\label{eqneweq117}h= a l+b l',
\end{eqnarray} where $q(l')=0$, by choosing a standard  hyperbolic basis of $\langle l,h\rangle\subset H^2(X,\mathbb{Q})$ equipped with the Beauville-Bogomolov form. We then have
\begin{eqnarray}
\label{eqactionutile}
h^{2n-2}\sigma_X= c_2l^{n-1}{l'}^{n-1}\sigma_X\,\,{\rm in}\,\,H^{4n-2}(X,\mathbb{C})
\end{eqnarray}
This  follows indeed from the Beauville-Fujiki relation
$$\alpha^{n+1}=0\,\,{\rm in }\,\,H^{2n+2}(X,\mathbb{C})$$
 when $q(\alpha)=0$, which by differentiation gives
$$\alpha^{n}\sigma=0\,\,{\rm in} \,\,H^{2n+2}(X,\mathbb{C})\,\, {\rm when} \,\,q(\alpha)=0,\,q(\sigma)=0,\,q(\alpha,\sigma)=0.$$
Applying this last equality to $\alpha=l$ or $l'$ and $\sigma=\sigma_X$, we get that
$l^n\sigma_X=0$  and $(l')^n\sigma_X=0$, which, using (\ref{eqneweq117}),  implies (\ref{eqactionutile}) with $c_2=\binom{2n-2}{n-1}a^{n-1}b^{n-1}$.
Combining (\ref{eqactionutile}) and (\ref{eqactionutile0}), (\ref{eqaction}) implies (\ref{eqaction2}) since, by  (\ref{eqactionutile0}), we have $[\Gamma_\pi]^*(h^{n-1}\sigma_X)=[\Gamma_\pi]^*({l'}^{n-1}\sigma_X)$.

We turn to the proof of (\ref{eqaction}). Let $U\subset B$ be the open set where $B$ is smooth and over which $\pi$ is smooth, and $X_U:=\pi^{-1}(U)$. It is a general fact due to Deligne \cite{deligne} that, if $V$ is a dense Zariski open set of a smooth projective variety $X$,  the restriction map $H^{2}(X,\mathbb{C})\rightarrow H^2(V,\mathbb{C})$ is injective  on the space
$H^{2,0}(X)$. (More precisely, Deligne shows that the kernel of the restriction map is a Hodge substructure of Hodge coniveau $\geq 1$.)  It thus suffices to prove (\ref{eqaction}) after restriction to the Zariski open set $X_U$. It also suffices to prove
the equality in $H^0(X_U,\Omega_{X_U}^2)$ because
the restriction map
$H^{2,0}(X)\rightarrow H^2(X_U)$ factors through the restriction map
$$H^{0}(X,\Omega_X^2)=H^{0}(X,\Omega_X^2)^{\rm closed}\rightarrow H^0(X_U,\Omega_{X_U}^2)^{\rm closed}.$$
We now use the fact that the fibration $\pi$ is Lagrangian, which means that, on $X_U$,
\begin{eqnarray}
\label{eqpoursigmaleray} \sigma_X\in \Gamma(\pi^*\Omega_U\wedge \Omega_{X_U}),
\end{eqnarray}

We next observe that, using the condition $H^0(\widetilde{B},\Omega_{\widetilde{B}}^2)=0$, in order to prove the equality (\ref{eqaction}) of $(2,0)$-forms on $X_U$, it suffices to prove
the corresponding  equality in the quotient $\Omega_U\otimes R^0\pi_*\Omega_{X_U/U}$ of $R^0\pi_*(\pi^*\Omega_U\wedge \Omega_{X_U})$.
The computation is now local and fiberwise. Indeed, project $\sigma_X$ to a section
$\tilde{\sigma}_X$ of $\pi^*\Omega_U\otimes \Omega_{X_U/U}$, and $h$ to a section
$\tilde{h}$ of
$R^1\pi_*\Omega_{X_U/U}$. Then the relative correspondence $\Gamma_{\pi,rel}\in {\rm CH}^1(X_U\times_U X_U)$ induces a morphism

\begin{eqnarray}
\label{eqgammarel} \Gamma_{\pi,rel}^*\circ \tilde{h}^{n-1}: R^0\pi_*\Omega_{X_U/U}\stackrel{\tilde{h}^{n-1}}{\rightarrow} R^{n-1}\pi_*\Omega_{X_U/U}^n\stackrel{\Gamma_{rel}^*}{\rightarrow}  R^0\pi_*\Omega_{X_U/U},
\end{eqnarray}
which tensored by $\Omega_U$  will induce the sheaf-theoretic  version
\begin{eqnarray}\label{eqlocalversion}([\Gamma_\pi]^*\circ h^{n-1})_{\rm loc}:\Omega_U\otimes R^0\pi_*\Omega_{X_U/U}\rightarrow \Omega_U\otimes R^0\pi_*\Omega_{X_U/U}
\end{eqnarray}
of the  morphism  $[\Gamma_\pi]^*\circ h^{n-1}$.
It thus only remains to prove that the morphism $\Gamma_{\pi,rel}^*\circ \tilde{h}^{n-1}$ of (\ref{eqgammarel}) is a nonzero rational multiple of the identity. This is a fiberwise statement  which is a standard fact about duality in abelian varieties.
\end{proof}
\begin{rema}\label{rematardivedu7sept} {\rm If we drop the assumption that $H^0(\widetilde{B},\Omega_{\widetilde{B}}^2)=0$, the argument above still proves that  (\ref{eqaction}) holds modulo $\pi^*H^0(\widetilde{B},\Omega_{\widetilde{B}}^2)$. This will be used in next section.}
\end{rema}
\begin{proof}[Proof of Theorem \ref{theobasic}] Let $X$ be a projective hyper-K\"ahler manifold with $\rho(X)=2$ admitting a Lagrangian fibration $\pi:X\rightarrow B$ with Lagrangian class $l$. As  $\rho(X)=2$, $X$ has two isotropic  classes in ${\rm NS}(X)$ (up to a scalar), which we call $l$ and $l'$.
As a consequence of  fundamental results of Huybrechts in \cite{huybrechts}, Riess proved the following:
\begin{theo} \cite[theorem 3.3]{riess}\label{theoriess}  Let $l'$ be an isotropic class on a projective hyper-K\"ahler manifold $X$ of dimension $2n$, with $q(l',\omega)>0$ for any K\"ahler form on $X$. Then there exists
a cycle $R \in{\rm CH}^{2n}(X\times X)$ with inverse $R^{-1}\in {\rm CH}^{2n}(X\times X)$ such that $R^*$ acts as an automorphism of ${\rm CH}(X)$ preserving
  the intersection product, the action of $R^*$ on $H^2(X,\mathbb{Q})$ preserves the Beauville-Bogomolov form  $q_X$, and
$R^*(l')$ belongs to the boundary of the birational K\"ahler cone of $X$.
\end{theo}
We are now faced with two possibilities: either

(a) $R^*(l')$ is proportional to $l'$, hence $l'$
belongs to the boundary of the birational K\"ahler cone of $X$, or

 (b)   $R^*(l')$ is proportional to $l$.

 Let us examine these possibilities separately
in the two cases \ref{condSYZ} and \ref{condtheta} of Theorem \ref{theobasic}.

\vspace{0.5cm}

{\bf  Case \ref{condSYZ}}: {\it $X$ satisfies the SYZ conjecture \ref{conjlag}.}   We assume  that we are in situation (a);
the situation (b) is treated below. The class $l'$, being in  the boundary of the birational K\"ahler cone of $X$, provides the Lagrangian class $l'_{X'}=(\psi^{-1})^* l'$ for  some  Lagrangian fibration on some hyper-K\"ahler manifold $X'$ which is birational to $X$ via a rational map
$\psi: X\dashrightarrow X'$. We apply the previous construction to $X'$ and $l'$, which provides a correspondence
$\Gamma_{\pi'}\in{\rm CH}^{n+1}(X'\times X')$ such that
\begin{eqnarray}
\label{eqGammalprime} [\Gamma_{\pi'}]^*({h'}^{2n-2}\sigma_{X'})=c'_1 (l'_{X'})^{n-1}\sigma_{X'},
\end{eqnarray}
where $h'=(\psi^{-1})^* h$.

By the work of Huybrechts \cite{huybrechts}, there exist  cycles $R_\psi\in {\rm CH}^{2n}(X\times X'),\,R_\psi^{-1}\in {\rm CH}^{2n}(X'\times X)$ which act as $\psi^*$, $(\psi^{-1})^*$ on degree $2$ cohomology   and  are constructed by adding    corrections
to the graphs of $\psi$ and $\psi^{-1}$, and whose action on cohomology is compatible with cup-product. Let

$$\Gamma_{\pi',X}:=R_\psi^{-1}\circ \Gamma_{\pi'}\circ R_\psi\in {\rm CH}^{n+1}(X\times X).$$
Then (\ref{eqGammalprime}), together with the compatibility of the action of $[R_\psi]$ with cup-product which gives $[R_\psi^{-1}]^*({h}^{2n-2}\sigma_{X})={h'}^{2n-2}\sigma_{X'},\,[R_\psi]^*((l'_{X'})^{n-1}\sigma_{X'})={l'}^{n-1}\sigma_X$,
shows that

\begin{eqnarray}
\label{eqGammalprimeprime} [\Gamma_{\pi',X}]^*(h^{2n-2}\sigma_{X})=c'_1 {l'}^{n-1}\sigma_{X},
\end{eqnarray}

Let now $$\Gamma_{\rm lef}:=\Gamma_{\pi',X}\circ {^t\Gamma_l} \in {\rm CH}^2(X\times X).$$
The proof of the Lefschetz standard conjecture in this case concludes, using Remark \ref{remaintro}, with the following lemma.
\begin{lemm} One has $[\Gamma_{\rm lef}]^*(h^{2n-2}\sigma_X)=\mu \sigma_X$ with $\mu\not=0$.
\label{lelefcal}
\end{lemm}
\begin{proof} One has $[\Gamma_{\rm lef}]^*(h^{2n-2}\sigma_X)=[{^t\Gamma}_\pi]^*([\Gamma_{\pi',X}]^*(h^{2n-2}\sigma_{X}))$ and by (\ref{eqGammalprimeprime}), this is equal to $c'_1[{^t\Gamma}_\pi]^*( {l'}^{n-1}\sigma_{X})$, so we have to prove that
$[{^t\Gamma}_\pi]^*( {l'}^{n-1}\sigma_{X})\not=0$. It clearly suffices to prove that
\begin{eqnarray}
\label{eqnonul} \langle [{^t\Gamma}_\pi]^*( {l'}^{n-1}\sigma_{X}),h^{2n-2}\overline{\sigma_X}\rangle_X\not=0.\end{eqnarray}
But $\langle [{^t\Gamma}_\pi]^*( {l'}^{n-1}\sigma_{X}),h^{2n-2}\overline{\sigma_X}\rangle_X=\langle {l'}^{n-1}\sigma_{X},[\Gamma_\pi]^*(h^{2n-2}\overline{\sigma_X})\rangle_X$, which by Proposition \ref{prodebaseactGamma}, equals
$c_1\langle {l'}^{n-1}\sigma_{X},l^{n-1}\overline{\sigma_X}\rangle_X$.
Thus the result follows from $\int_X(ll')^{n-1} \sigma_{X}\overline{\sigma_X}\not=0$, which follows from the second Hodge-Riemann relations, since, by (\ref{eqactionutile}), this is a nonzero multiple of   $\int_Xh^{2n-2} \sigma_{X}\overline{\sigma_X}$.
\end{proof}

\vspace{0.5cm}

{\bf Case \ref{condtheta}}: {\it   There is an effective divisor $\Theta$ on $X$ which satisfies $q(\theta)<0$.} In this case, we claim that $l'$ does not
belong to the boundary of the birational K\"ahler cone of $X$ (see also \cite{sacca}, where a similar situation is studied for the LSV variety constructed in \cite{lsv}). Indeed, $l'= a \theta+bl$ with
$$q(l')=0,\,\,q(l',l)=aq(\theta,l)>0,\,\,q(l,\theta)>0,$$
 and this implies that $a>0$ and
 $$a^2 q(\theta)+2abq(l,\theta)=0.$$
 As $q(\theta)<0$, we find that $b>0$ and
 thus $$q(\theta, l')=aq(\theta)+bq(\theta,l)=-bq(\theta,l)<0. $$
 As $\theta$  is the class of an effective divisor, this implies the claim.
 It follows from the claim that we must have $R^*l'=l$, that is, we are in situation (b).

 \vspace{0.5cm}

 We can  now in both cases assume that we are in situation (b) and conclude the proof of Theorem \ref{theobasic}. We  have $l'=(R^{-1})^*l$ and we know that $l$ is a Lagrangian class on $X$ and $R^{-1}\in{\rm CH}^{2n}(X\times X)$ is a correspondence whose action in cohomology preserves the cup-product and is the identity on $H^{2,0}(X)$ (see Theorem \ref{theoriess}). Let $\Gamma_\pi\in{\rm CH}^{n+1}(X\times X)$ be constructed as before and let
 $$\Gamma_{l'}:=R^{-1}\circ \Gamma_\pi\circ R.$$
 Then, by Proposition \ref{prodebaseactGamma}, we have
 $[\Gamma_{\pi}]^*(h^{2n-2}\sigma_{X})=c_1 l^{n-1}\sigma_{X}$, hence,  if $h'=(R^{-1})^*h$, we have, as in (\ref{eqGammalprime})

\begin{eqnarray}
\label{eqGammalprimesituationb} [\Gamma_{l'}]^*({h'}^{2n-2}\sigma_{X})=c_1 {l'}^{n-1}\sigma_{X}.
\end{eqnarray}
We can thus  apply the same argument as before, replacing $\Gamma_{\pi',X}$ by $\Gamma_{l'}$ and defining
 $$\Gamma_{\rm lef}:=\Gamma_{l'}\circ {^t\Gamma_l} \in {\rm CH}^2(X\times X).$$
 The analog of Lemma \ref{lelefcal} is proved as before, and this establishes the Lefschetz standard conjecture in this case.
\end{proof}
In the case of dimension $4$, we have the following alternative proof of Theorem \ref{theobasic} giving a slightly different statement.
\begin{prop} Let $X$ be a hyper-K\"ahler fourfold admitting a Lagrangian fibration $\pi: X\rightarrow B$. Assume  there exists a
uniruled divisor $\Theta\subset X$ such that $q(\theta,l)\not=0$. Then
$X$ satisfies the Lefschetz standard conjecture in degree $2$.
\end{prop}
\begin{proof} Let $\widetilde{\Theta}$ be a desingularization of $\Theta$ and
let $p:\widetilde{\Theta}\rightarrow S$ be the maximal rationally connected fibration of $\widetilde{\Theta} $. Then $S$ is a surface since
$\widetilde{\Theta}$ has a nonzero holomorphic $2$-form pulled-back from $X$. We claim  that the condition $q(\theta,l)\not=0$ implies that
$\pi_{\mid \Theta}: {\Theta}\rightarrow B$ is surjective. Indeed, if $\pi_{\mid \Theta}: {\Theta}\rightarrow B$ is not  surjective, then
$l^2\theta^2=0$ and this implies $q(l,\theta)=0$ by the following argument. The Beauville-Fujiki relations give, for a nonzero rational constant $c_X$,
$$(\alpha l+\beta \theta)^4=c_X q(\alpha l+\beta \theta)^2=c_X(2\alpha\beta q( l,\theta)+\beta^2 q(\theta))^2,$$
and comparing the coefficients in $\alpha^2\beta^2$ we get that $l^2\theta^2$ is proportional to $q(l,\theta)^2$, which proves the claim.
 We  then  construct as in \cite{voisintriangle} a surface decomposition for $X$ using
the ``sum map up to isogeny" $\mu:X\times_B X\dashrightarrow X$, for which we only need to use a surface $T\subset X$ providing a multisection of $\pi$.
The surface decomposition  is given by the following construction: let
$$\Gamma=\widetilde{\Theta}\times_{B}\widetilde{\Theta},\,\,\psi=\mu_{\mid \widetilde{\Theta}\times_{B}\widetilde{\Theta}}:\Gamma\rightarrow X,$$
and $\phi:=(p,p)_{\mid\widetilde{\Theta}\times_{B}\widetilde{\Theta}}: \Gamma\rightarrow S\times S$.
The claim implies that $\psi$ is surjective, hence generically finite since ${\rm dim}\,\Gamma=4$.
It is proved in \cite{voisintriangle} that these maps satisfy
\begin{eqnarray}\label{eqpubaform} \psi^*\sigma_X=\phi^* ({\rm pr}_1^*\sigma_{S,1}+{\rm pr}_2^*\sigma_{S,2})
\end{eqnarray}
for some holomorphic forms $\sigma_{S,i}$ on $S$.
As $\psi_*\psi^*\sigma_X=({\rm deg}\,\psi) \sigma_X$, (\ref{eqpubaform})  implies the Lefschetz standard conjecture, using Proposition \ref{procharles}.
\end{proof}
\begin{rema}{\rm If $X$ has a second Lagrangian fibration, then it is likely that one can take for $\Theta$ the locus of singular fibers of $\pi'$, but to our knowledge,  it is not fully proved that this divisor is uniruled.}
\end{rema}
\subsection{More general coverings by Lagrangian varieties \label{secprooflagcov}}
This section is devoted to the proof of Theorem \ref{corodim4}. First of all, we have the following criterion for the Lefschetz standard conjecture to hold in degree $2$.

 \begin{prop}\label{proassezgencriterelef} Let $X$ be a smooth projective $2n$-fold with polarization $h_X$ and
 let $\Gamma\in{\rm CH}^{n+1}(X\times X)$ having the property that the intersection pairing $\langle\,,\,\rangle_X$ is nondegenerate on
 $[\Gamma]^* (h_X^{2n-2} H^2(X,\mathbb{Q})_{tr})\subset H^{2n}(X,\mathbb{Q})$.
 Then $X$ satisfies the Lefschetz standard conjecture for degree $2$ cohomology.
 \end{prop}
 \begin{proof} We consider the codimension $2$ cycle $\Gamma':= {\Gamma}\circ {^t\Gamma}\in {\rm CH}^2(X\times X)$. We claim that
 \begin{eqnarray}\label{eqnoudu7sep} [\Gamma']^*: h_X^{2n-2} H^2(X,\mathbb{Q})_{tr}\rightarrow H^2(X,\mathbb{Q})_{tr}\end{eqnarray}
 is an isomorphism. We note first that $[\Gamma']^*: h_X^{2n-2} H^2(X,\mathbb{Q})_{tr}\rightarrow H^2(X,\mathbb{Q})$ has image contained
 in $H^2(X,\mathbb{Q})_{tr}$, so that (\ref{eqnoudu7sep}) is well defined. This is because $[\Gamma']^*$ is a morphism of Hodge structures and by definition of transcendental
 cohomology, any Hodge substructure of $ H^2(X,\mathbb{Q})_{tr}$ which (after tensoring with $\mathbb{C}$) contains $ H^{2,0}(X)$ equals $ H^2(X,\mathbb{Q})_{tr}$. This implies that any morphism of Hodge structures from $ H^2(X,\mathbb{Q})_{tr}$ to a trivial Hodge structure is $0$, hence the composition $$[\Gamma']^*: h_X^{2n-2} H^2(X,\mathbb{Q})_{tr}\rightarrow H^2(X,\mathbb{Q})\rightarrow H^2(X,\mathbb{Q})_{tr}$$ is zero, where on the right, we see  $H^2(X,\mathbb{Q})_{tr}$ as the quotient of $ H^2(X,\mathbb{Q})$ by the space of rational  Hodge classes on $X$.
 Next, as both spaces in (\ref{eqnoudu7sep}) have the same dimension,  in order to prove the claim, it suffices to show that the intersection pairing $$(\alpha,\beta):=\langle [\Gamma']^*\alpha,\beta\rangle_X$$
 on $h_X^{2n-2} H^2(X,\mathbb{Q})_{tr}$ is nondegenerate. But we have by definition of $\Gamma'$
 $$\langle [\Gamma']^*\alpha,\beta\rangle_X=\langle [\Gamma]^*\alpha,[\Gamma]^*\beta\rangle_X$$
 and the pairing defined by the right hand side is nondegenerate by assumption. The claim implies the proposition since
 by restricting $\Gamma'$ to $ X \times\Sigma$ where $\Sigma\subset X$ is a surface complete intersection  of hypersurfaces of class $h_X$,
 we find that $[\Gamma]^*: H^{2,0}(\Sigma)\rightarrow H^{2,0}(X)$ is surjective, so that Proposition \ref{procharles} applies.
 \end{proof}
Assume now that a hyper-K\"ahler manifold $X$ admits a Lagrangian covering, given by  a diagram as in Definition \ref{defigoodcov}
\begin{eqnarray}
\label{eqdiaglag}
\begin{matrix}
 & L&\stackrel{\Phi}{\longrightarrow}& X
 \\
&\pi\downarrow& &
\\
&B& &
\end{matrix}
\end{eqnarray}
of smooth projective varieties, such that $\Phi$ is dominant and,  for a general point $b\in B$,
 $\Phi_{\mid L_b}$ is generically finite to its image, which is a Lagrangian subvariety  of $X$.

If  ${\rm dim}\,X=2n$, so ${\rm dim}\,L_b=n$ and ${\rm dim}\,B\geq n$,
 as the map $\Phi_{\mid L_b}$ is generically finite to its image, up to replacing $B$ by a complete intersection   of ample hypersurfaces, one can assume that ${\rm dim}\,L={\rm dim}\,X$ and $\Phi$ is dominant generically finite of degree $N$.
  We now  introduce  a relative  Poincar\'{e} divisor $d\in{\rm CH}^1(L\times_BL)$, that is, for general $b\in B$,
 $d_b\in{\rm Pic}(L_b\times L_b)$ has the property that
 the map $y\in L_b\mapsto d_{b,y}\in {\rm Pic}(L_b)$ induces an  isogeny ${\rm Alb}(L_b)\rightarrow {\rm Pic}^0(L_b)$. Using the fact that $\Phi_{\mid L_b}$ is generically finite on its image for a general $b\in B$, for an ample line bundle $H_X$ on $X$,
 $\Phi^*H_{X\mid L_b}$ is big hence gives a polarization on the Hodge structure on $H^1(X_b,\mathbb{Q})$, hence an isogeny $\eta:{\rm Alb}(X_b)\rightarrow {\rm Pic}^0(X_b)$. We thus have a natural rational map (which is a morphism over a Zariski open set of $B$)
 $$\rho=(\rho_1,\rho_2): L\times_BL\dashrightarrow {\rm Alb}(X/B)\times_B {\rm Pic}^0(X/B)$$
 where $\rho_1$ is an isogenous  version of the relative  Albanese map constructed as in the previous section   using a multisection of $\pi$ and $\rho_2=\eta\circ\rho_1)$, and we can pull-back under  $\rho$ a relative Poincar\'{e} divisor on ${\rm Alb}(X/B)\times_B {\rm Pic}^0(X/B)$. This provides the desired  codimension $1$ cycle  $d$ on $L\times_BL$, and
   pushing-forward $d$ via the inclusion $L\times_BL\hookrightarrow L\times L$,
 we get  a codimension $n+1$-cycle $\Gamma_L\in{\rm CH}^{n+1}(L\times L)$.
  Denote by $L^1H^{2,0}(L)\subset H^{2,0}(L)$ the Leray level
$$L^1H^{2,0}(L)={\rm Ker}\,(H^{2,0}(L)\rightarrow H^{2,0}(L_b)).$$ The fact that $\Phi:L\rightarrow X$ is a covering by Lagrangian varieties says equivalently that
$\Phi^*\sigma_X\in L^1H^{2,0}(L)$.
 \begin{lemm} For any  $\sigma\in L^1H^{2,0}(L)$, one has

\begin{eqnarray}\label{eqactionpourlagcov}  [\Gamma_{L}]^*(({\Phi}^*h_X^{n-1})\sigma)=\lambda \sigma\,\,{\rm mod}\,\,\pi^*H^{2,0}(B)
\end{eqnarray}
for some nonzero rational coefficient $\lambda$.
\end{lemm}
\begin{proof} The proof is essentially the same as the proof of  (\ref{eqaction}) of  Proposition \ref{prodebaseactGamma}. The geometric context is slightly different since we were working in Proposition \ref{prodebaseactGamma} with a family of abelian varieties. However, we can work above with the family of abelian varieties  ${\rm Alb}(X/B)$ which inherits a holomorphic  $2$-form. The proof of (\ref{eqaction}) did not use indeed the fact that the holomorphic $2$-form is nondegenerate.  Furthermore, as noticed in Remark \ref{rematardivedu7sept} the proof given there proves this equality in general,
modulo $\pi^*H^{2,0}(B)$.
\end{proof}
  Using the generically finite map  $\Phi:L\rightarrow X$, we  get a
 codimension $n+1$-cycle

\begin{eqnarray}\label{eqdefgammaXL} \Gamma_{L,X}:=(\Phi,\Phi)_*(\Gamma_{L}).
\end{eqnarray}

We have the following result
\begin{prop}\label{coroprovisoire} In the situation above, assume that  $\Phi_*(\pi^* H^{2,0}(B))=0$. Then
\begin{eqnarray}\label{eqevidente}  [\Gamma_{L,X}]^{*}(h_X^{n-1}\sigma_X)=\lambda'\sigma_X
\end{eqnarray}
for some nonzero coefficient $\lambda'$. Furthermore,
\begin{eqnarray}
\label{eqdualeevidente}[{ ^t\Gamma}_{L,X}]^{*}(h_X^{2n-2}\sigma_X)\not=0\,\,{\rm in} \,\,H^{2n}(X,\mathbb{C}).
\end{eqnarray}
\end{prop}
\begin{proof} We apply (\ref{eqactionpourlagcov})  to ${\Phi}^*\sigma_X $ and obtain
$$[\Gamma_{L}]^*({\Phi}^*(h_X^{n-1}\sigma_X))=\lambda {\Phi}^*\sigma_X+\pi^*\beta$$
for some $\beta\in H^{2,0}(B)$.
Pushing forward this equation to $X$ via $\Phi$, and using the facts that
$\Phi_*(\pi^* H^{2,0}(B))=0$, and
$\Phi_*\circ {\Phi}^*=D Id$,  where $D={\rm deg}\,\Phi$, we get
$$[\Gamma_{L,X}]^*(h_X^{n-1}\sigma_X)=D\lambda \sigma_X,$$
which proves (\ref{eqevidente}), with $\lambda'=D\lambda$.

We have
$$ \langle [{^t\Gamma}_{L,X}]^{*}(h_X^{2n-2}\sigma_X),h_X^{n-1}\overline{\sigma_X}\rangle_X=\langle h_X^{2n-2}\sigma_X,[\Gamma_{L,X}]^{*}h_X^{n-1}\overline{\sigma_X}\rangle_X,
$$
and this is nonzero by (\ref{eqevidente}), which proves (\ref{eqdualeevidente}).
\end{proof}

We now prove Theorem \ref{corodim4}  stated in the introduction,  concerning the case of dimension $4$.
\begin{proof}[Proof of Theorem  \ref{corodim4}] Let $X$ be a hyper-K\"ahler fourfold with $\rho(X)=1$ admitting  a covering by Lagrangian surfaces
\begin{eqnarray}\label{eqlagcovnewnew}
\begin{matrix}
 & L&\stackrel{\Phi}{\longrightarrow}& X
 \\
&\pi\downarrow& &
\\
&B& &
\end{matrix},
\end{eqnarray}
where we can assume that $B$ is a surface, up to replacing $B$ by a complete intersection in $B$.
We perform the construction described above, which provides us with the cycle $\Gamma_{L,X}\in {\rm CH}^3(X\times X)$.
Suppose first that   $\Phi_*(\pi^*H^{2,0}(B))\not=0$ in $H^{2,0}(X)$.  Then,  the correspondence $\Gamma\in {\rm CH}^2(X\times B)$ given by the Lagrangian cover
(\ref{eqlagcovnewnew})
where $B$ is a surface,
directly solves the Lefschetz standard conjecture for degree $2$ cohomology of $X$, since dually
$[\Gamma]^*:H^{2,0}(B)\rightarrow H^{2,0}(X)$ is nonzero, hence surjective, so Proposition \ref{procharles} applies in this case.

We can thus assume that $\Phi_*(\pi^*H^{2,0}(B))=0$ in $H^{2,0}(B)$, so that  Proposition \ref{coroprovisoire} applies, and we conclude that the morphism of Hodge structures
\begin{eqnarray} [\Gamma_{L,X}]^*: h_X^2H^2(X,\mathbb{Q})_{tr}\rightarrow H^4(X,\mathbb{Q})
\end{eqnarray}
is not equal to $0$.

Assume now that $X$ is very general in moduli, which is part of the  assumptions of Theorem \ref{corodim4}. Then, by the local surjectivity of the period map,  the Mumford-Tate group of the Hodge structure on $H^2(X,\mathbb{Q})_{tr}$ is the  whole orthogonal group of the Beauville-Bogomolov quadratic form. We now have
\begin{lemm} \label{lenouveaupourdim4}If the Hodge structure on $H^2(X,\mathbb{Q})_{tr}$ has maximal Mumford-Tate group, any morphism of Hodge
structure $\phi: H^2(X,\mathbb{Q})_{tr}\rightarrow H^4(X,\mathbb{Q})$ is given by cup-product by an element
$h\in {\rm NS}(X)_\mathbb{Q}$.
\end{lemm}
\begin{proof} We first start with the following easy
\begin{lemm}\label{remanew111}  If $X$ is hyper-K\"ahler of  dimension $4$, the Hodge substructure on
the orthogonal complement  $H^4(X,\mathbb{Q})^{\perp SH^2}$ of ${\rm Sym}^2H^2(X,\mathbb{Q})$ in $H^4(X,\mathbb{Q})$ consists   of Hodge  classes.
\end{lemm}
\begin{proof}  Indeed, the cup-product map
$$\sigma_X\cup:H^{1}(X,\Omega_X^{})\rightarrow H^{1}(X,\Omega_X^{3})=H^{3,1}(X)$$
is an isomorphism, hence the Hodge structure on
 $H^4(X,\mathbb{Q})^{\perp SH^2}$ has $H^{4,0}=0,\,H^{3,1}=0$.
 \end{proof}
We conclude that the Hodge structure on $H^4(X,\mathbb{Q})$
decomposes as
\begin{eqnarray}\label{eqdecomp} H^4(X,\mathbb{Q})={\rm Sym}^2H^2(X,\mathbb{Q})_{tr}\oplus {\rm NS}(X)_\mathbb{Q}\otimes H^2(X,\mathbb{Q})_{tr}\oplus K,
\end{eqnarray}
where $K$ consists of Hodge classes coming either from  ${\rm Sym}^2{\rm NS}(X)_\mathbb{Q}$ or from $H^4(X,\mathbb{Q})^{\perp SH^2}$.
The fact that the Mumford-Tate group of the Hodge structure on $H^2(X,\mathbb{Q})_{tr}$ is the whole orthogonal group implies that
the Hodge structure ${\rm Sym}^2H^2(X,\mathbb{Q})_{tr}$ is the sum of the $1$-dimensional vector space $\mathbb{Q}c_{BB}$ generated by the   Hodge class $c_{BB}$ giving the Beauville-Bogomolov
quadratic form, and an irreducible Hodge structure ${\rm Sym}^2H^2(X,\mathbb{Q})_{tr}^0$ of Hodge niveau $4$. In particular, there is no nontrivial morphism of Hodge structures from $H^2(X,\mathbb{Q})_{tr}$ to ${\rm Sym}^2H^2(X,\mathbb{Q})_{tr}^0$. As there is no nontrivial morphism of Hodge structures from
$H^2(X,\mathbb{Q})_{tr}$ to a trivial Hodge structure, it follows that
the morphism $\phi$ takes value in ${\rm NS}(X)_\mathbb{Q}\otimes H^2(X,\mathbb{Q})_{tr}$, which, as a Hodge structure,  is a sum of copies of $ H^2(X,\mathbb{Q})_{tr}$. As the Mumford-Tate group of the Hodge structure on $H^2(X,\mathbb{Q})_{tr}$ is the whole orthogonal group, there are no nontrivial endomorphisms of the Hodge structure on $H^2(X,\mathbb{Q})_{tr}$, hence $\phi$ is given by cup-product by an element of ${\rm NS}(X)_\mathbb{Q}$.
\end{proof}
We now conclude the proof. As
 $\rho(X)=1$, the divisor class $h$ provided by Lemma \ref{lenouveaupourdim4} is a nonzero multiple of $h_X$. Thus we have by the nonvanishing of  $ [\Gamma_{L,X}]^*$,
 $$[\Gamma_{L,X}]^*\sigma_X=\lambda'' h_X\sigma_X,$$
 for some nonzero coefficient $\lambda''$.
 It then follows from  the second  Hodge-Riemann relations that
 $$\langle [\Gamma_{L,X}]^*\sigma_X,[\Gamma_{L,X}]^*\overline{\sigma_X}\rangle\not=0.$$

But then the intersection pairing on ${\rm Im}\,([\Gamma_{L,X}]^*:h_X^2H^2(X,\mathbb{Q})_{tr}\rightarrow H^2(X,\mathbb{Q})_{tr})$ is nondegenerate, so that Proposition \ref{proassezgencriterelef} applies.
\end{proof}

\section{Zero-cycles \label{SEC2}}

\subsection{Zero-cycles and  the Lefschetz standard conjecture for $H^2$ \label{seclef}}

We wish to discuss in this section  the link between Bloch-Beilinson type  filtrations on the group
${\rm CH}_0(X)$ of $0$-cycles and the Lefschetz standard conjecture in degree $2$.
Given a smooth projective variety $X$, we define as usual
$$F^1{\rm CH}_0(X)={\rm CH}_0(X)_{hom},\,F^2{\rm CH}_0(X):={\rm Ker}\,({\rm alb}_X: {\rm CH}_0(X)_{hom}\rightarrow ({\rm Alb}\,X)_\mathbb{Q}).$$
Several proposals have been made in order to construct
$F^3{\rm CH}_0(X)\subset F^2{\rm CH}_0(X)$.
In \cite{saito}, Sh. Saito  consider a subgroup that we will denote by
$F^3_{dR}{\rm CH}_0(X)$ and is (essentially) defined as follows.

\begin{eqnarray}
\label{F3dR}  F^3_{dR}{\rm CH}_0(X)=\langle {\rm Im}\,(\Gamma_*:F^2{\rm CH}_0(Y)\rightarrow {\rm CH}_0(X))\\
\nonumber
{\rm for\,\,all}\,\,\Gamma\,\,{\rm such\,\,that}\,\,0=\Gamma^*:H^{2,0}(X)\rightarrow H^{2,0}(Y)\rangle.
\end{eqnarray}
\begin{rema} \label{remapourseclef}{\rm  In this definition, $\Gamma\in {\rm CH}^{{\rm dim}\,X}(Y\times X)$ and all smooth projective varieties $Y$ are considered, but  one can also restrict to  surfaces
$Y$  by the following lemma.}
\end{rema}
\begin{lemm} \label{lepourseclef} Let $Y$ be a smooth projective variety. Then
$$F^2{\rm CH}_0(Y)=\langle {\rm Im}\,(\Gamma_*:F^2{\rm CH}_0(Z)\rightarrow {\rm CH}_0(X)),\,\,{\rm where}\,\,Z\,\,{\rm is\,\,a\,\,surface}\rangle.$$
\end{lemm}
\begin{proof} Indeed, we can assume ${\rm dim}\,Y\geq 2$. Any $0$-cycle $z$ of $Y$ is supported on a surface
$Z\subset Y$ which is a smooth complete intersection of ample hypersurfaces. By Lefschetz theorem on hyperplane sections, ${\rm Alb}\,Z={\rm Alb}\,Y$, so if
${\rm alb}_Y(z)=0$, ${\rm alb}_Z(z)=0$.
\end{proof}
The definition of $F^3_{dR}{\rm CH}_0$ given in (\ref{F3dR}) is essentially motivated by the axioms of the Bloch-Beilinson filtration
(see \cite{bloch}). Indeed,
suppose $Y$ is a surface, and $\Gamma\in {\rm CH}^{{\rm dim}\,X}(Y\times X)$ is a correspondence  satisfying  the property that
$\Gamma^*:H^{2,0}(X)\rightarrow H^{2,0}(Y)$ is zero. Then
one gets by K\"unneth decomposition on the surface  $Y$ (see \cite{murre}) that the class $[\Gamma]$ of $\Gamma$  decomposes as
\begin{eqnarray}\label{eqdecompdu7sep}[\Gamma]=\sum_i[\Gamma_i\times \Gamma'_i]+[R_1]+[R_2]\end{eqnarray} for some cycles
$\Gamma_i$, resp. $\Gamma'_i$ in $Y$, resp. $X$, such that ${\rm codim}\,\Gamma_i+{\rm codim}\,\Gamma'_i=2$, where the cycle $R_1$ is supported on $C_1\times X$ for a curve $C_1\subset Y$, and the cycle $R_2$ is supported on $Y\times C_2$ for a curve $C_2\subset X$.
A decomposable cycle $\sum_i\Gamma_i\times \Gamma'_i$ and a cycle $R_1$ supported on $C_1\times X$ act trivially on $0$-cycles of degree $0$; the cycle $R_2$ acts trivially on $0$-cycles of $Y$ which are of degree $0$ and annihilated by ${\rm alb}_Y$, that is, which belong to $F^2{\rm CH}_0(Y)$. Hence we get
$$\Gamma_*=(\Gamma-\sum_i\Gamma_i\times \Gamma'_i-R_1-R_2)_*: F^2{\rm CH}_0(Y)\rightarrow F^2{\rm CH}_0(X).$$
As the cycle $\Gamma-\sum_i\Gamma_i\times \Gamma'_i-R_1-R_2$ is cohomologous to $0$ by (\ref{eqdecompdu7sep}), the image $\Gamma_*(F^2{\rm CH}_0(Y))$ must be contained in the $F^3$ level of the Bloch-Beilinson filtration.

We now turn to a different  definition, introduced in \cite{voisinanali}, also a candidate for the third level of the Bloch-Beilinson filtration.

\begin{eqnarray}
\label{F3BB}  F^3_{BB}{\rm CH}_0(X):=\cap_{\Sigma,\Gamma} {\rm Ker}\,(\Gamma_*:{\rm CH}_0(X)\rightarrow {\rm CH}_0(\Sigma)),
\end{eqnarray}
 where the intersection is over all smooth projective surfaces $\Sigma$ and  correspondences $\Gamma\in{\rm CH}^2(X\times \Sigma)$.
Again, this definition is dictated by the Bloch-Beilinson axioms since $\Gamma_*$ should preserve the Bloch-Beilinson filtration and for a Bloch-Beilinson filtration $F$, $F^3{\rm CH}_0(\Sigma)=0$ for any smooth projective surface.
\begin{rema}{\rm \label{remanewpoursurf}  It is obvious from the definition that $F^3_{BB}{\rm CH}_0(S)=0$ for a smooth projective surface $S$.}
\end{rema}

Let us recall the statement of  the Bloch conjecture for $0$-cycles on a surfaces (see \cite{bloch}):
\begin{conj} \label{conjbloch}(Bloch) Let $\Gamma\in{\rm CH}^2(S\times T)$ where $T$ and $S$ are smooth projective surfaces.
Then if $\Gamma^*:H^{2,0}(T)\rightarrow H^{2,0}(S)$ is zero, $\Gamma_*:F^2{\rm CH}_0(S)\rightarrow {\rm CH}_0(T)$ is zero.
\end{conj}
We now have the following implications.
\begin{prop}  \label{propbloch} The inclusion $F^3_{dR}{\rm CH}_0(X)\subset F^3_{BB}{\rm CH}_0(X)$ holds for any $X$ if and only if the Bloch conjecture \ref{conjbloch} holds for $0$-cycles on surfaces.
\end{prop}
\begin{proof}  Using Remark \ref{remapourseclef}, $F^3_{dR}{\rm CH}_0(X)$ is generated by cycles $\Gamma_*z$, where
 $S$ is  a smooth projective surface, $\Gamma\in {\rm CH}^{{\rm dim}\,X}(S\times X)$
 has the property that
 $\Gamma^*:H^{2,0}(X)\rightarrow H^{2,0}(S)$ is zero, and $z\in F^2{\rm CH}_0(\Sigma)$. Let $w=\Gamma_*z$ be such a cycle and let now
 $\Gamma'\in {\rm CH}^2(X\times \Sigma)$ be a correspondence, where $\Sigma$ is  a smooth projective surface. Then
 $\Gamma'_*w=(\Gamma'\circ \Gamma)_*z$ and the correspondence
 $$\Gamma'\circ \Gamma\in {\rm CH}^2(S\times \Sigma)$$ has the property that
$(\Gamma'\circ \Gamma)^*:H^{2,0}(\Sigma)\rightarrow H^{2,0}(S)$ is zero. Hence if Bloch conjecture \ref{conjbloch} holds,
$(\Gamma'\circ \Gamma)_*z=0$ in ${\rm CH}_0(\Sigma)$. Hence $w\in F^3_{BB}{\rm CH}_0(X)$.

 Conversely, assume the inclusion $F^3_{dR}{\rm CH}_0(X)\subset F^3_{BB}{\rm CH}_0(X)$ for any smooth projective $X$. Let $S,\,T$ be smooth projective surfaces and let $\Gamma\in{\rm CH}^2(S\times T)$ such that $\Gamma^*: H^{2,0}(T)\rightarrow H^{2,0}(S)$ vanishes. Then by definition,
$$\Gamma_*(F^2{\rm CH}_0(S))\subset F^3_{dR}{\rm CH}_0(T)$$
and $F^3_{dR}{\rm CH}_0(T)\subset F^3_{BB}{\rm CH}_0(T)$ by our assumption. As $ F^3_{BB}{\rm CH}_0(T)=0$ by Remark \ref{remanewpoursurf}, we get $\Gamma_*(F^2{\rm CH}_0(S))=0$.
 \end{proof}
\begin{prop} \label{proplef}  The inclusion $ F^3_{BB}{\rm CH}_0(X) \subset F^3_{dR}{\rm CH}_0(X)$ holds for  $X$ if $X$ satisfies the Lefschetz standard conjecture for degree $2$ cohomology.
\end{prop}
\begin{proof} Assuming the Lefschetz standard conjecture for degree $2$ cohomology on $X$, there exist
a surface $\Sigma$ and correspondences
$\Gamma\in {\rm CH}^2(X\times \Sigma),\,\Gamma'\in {\rm CH}^{{\rm dim}\,X}(\Sigma\times X)$ such that
$\Gamma'\circ \Gamma\in {\rm CH}^{{\rm dim}\,X}(X\times X)$ acts as the identity on $H^{2,0}(X)$.
For any $z\in F^2{\rm CH}_0(X)$,
\begin{eqnarray}
\label{eqzerocycleF3} z-(\Gamma'\circ \Gamma)_*z\in F^3_{dR}{\rm CH}_0(X)
\end{eqnarray}
 since the correspondence
$\Delta_X-\Gamma'\circ \Gamma$ acts trivially on $H^{2,0}(X)$.
Let now $z\in F^3_{BB}{\rm CH}_0(X)$. Then $\Gamma_*z=0$ in ${\rm CH}_0(\Sigma)$, hence
$(\Gamma'\circ \Gamma)_*z=0$ in ${\rm CH}_0(X)$. Hence
$ z=z-(\Gamma'\circ \Gamma)_*z\in F^3_{dR}{\rm CH}_0(X)$ by (\ref{eqzerocycleF3}).
\end{proof}
\begin{theo}\label{propeqlef} The Lefschetz standard conjecture for degree $2$ cohomology holds on $X$ if and only if there exists a
surface $\Sigma\stackrel{j}{\hookrightarrow} X$ such that $j_*:{\rm CH}_0(\Sigma)\rightarrow {\rm CH}_0(X)/F^3_{dR}{\rm CH}_0(X)$ is surjective.
\end{theo}
\begin{proof} The ``only if'' direction follows from the arguments given in the previous proof. Indeed, we proved that the Lefschetz standard conjecture implies  the equality
(\ref{eqzerocycleF3}) for $z\in F^2{\rm CH}_0(X)$, where the cycle  $(\Gamma'\circ \Gamma)_*z$ is supported on a surface
 $S\stackrel{j}{\hookrightarrow} X$, namely, we can take  for $S$ the image  in $X$ of ${\rm Supp}\,\Gamma'$ by the second projection.
 This equality thus says  $ F^2{\rm CH}_0(X)/F^3_{dR}{\rm CH}_0(X)$ is contained in the image of  $j_*:{\rm CH}_0(S)\rightarrow  {\rm CH}_0(X)/F_{dR}^3{\rm CH}_0(X)$. If furthermore $S$ contains an ample complete intersection surface, then the map
 $j_*:{\rm CH}_0(S)\rightarrow  {\rm CH}_0(X)/F^2{\rm CH}_0(X)$ is also surjective because the Lefschetz hyperplane section theorem implies that $j_*:{\rm Alb}\,S\rightarrow {\rm Alb}\,X$ is surjective. Hence
 $j_*:{\rm CH}_0(S)\rightarrow  {\rm CH}_0(X)/F_{dR}^3{\rm CH}_0(X)$ is surjective.

We now prove the reverse implication. Assume that there is a surface $\Sigma\stackrel{j}{\hookrightarrow } X$ such that $j_*:{\rm CH}_0(\Sigma)\rightarrow {\rm CH}_0(X)/F^3_{dR}{\rm CH}_0(X)$ is surjective. Then for a general point
$x\in X$, there exist a smooth projective  surface $Y$,  a correspondence $\Gamma\in{\rm CH}^{{\rm dim}\,X}(Y\times X)$,
such that $\Gamma^*:H^{2,0}(X)\rightarrow H^{2,0}(Y)$ is zero, a $0$-cycle $z_x\in F^2{\rm CH}_0(Y)$  and a $0$-cycle $z'_x\in {\rm CH}_0(\Sigma)$ such that
$x=\Gamma_*(z_x)+z'_x$ in ${\rm CH}_0(X)$.

As it follows from Bloch-Srinivas arguments \cite{blochsrinivas}, these data can be spread over a finite extension of the function
field of $X$, and this produces

(1) a smooth projective variety  $\mathcal{Y}$, a dominant morphism $f: \mathcal{Y}\rightarrow X$ of relative dimension $2$  and a
cycle
$\Gamma \in{\rm CH}^{{\rm dim}\,X}(\mathcal{Y}\times X)$, such that
$\Gamma_x^*:H^{2,0}(X)\rightarrow H^{2,0}(\mathcal{Y}_x)$ is $0$ for a general point $x\in X$;

(2)  two  cycles
$$\mathcal{Z}\in {\rm CH}^{2}(\mathcal{Y}),\,\,\mathcal{Z}'\in {\rm CH}^{2}(X\times \Sigma),$$
such that  $\mathcal{Z}_x\in F^2{\rm CH}_0(\mathcal{Y}_x)$ for a general point $x\in X$, and for some integer $N\not=0$
\begin{eqnarray}
\label{eqpourdiag}  N\Delta_X  =(f,\Gamma)_*( \mathcal{Z})+(Id_X,j)_*\mathcal{Z}'+W\,\,{\rm in}\,\,{\rm CH}^{{\rm dim}\,X}(X\times X)
\end{eqnarray}
where $W\in {\rm CH}^{{\rm dim}\,X}(X\times X)$ is a cycle which is supported on $D\times X$, for some divisor $D\subset X$.
We now examine how the various cycles appearing in (\ref{eqpourdiag}) act on $H^{2,0}(X)$.
One has $[W]^*=0$ on $H^{2,0}(X)$ because $W$ is supported on $D\times X$, for some divisor $D\subset X$.
Next we claim that the cycle $A:=(f,\Gamma)_*( \mathcal{Z})\in {\rm CH}^{{\rm dim}\,X}(X\times X)$ satisfies
 $[ A]^*(\eta)=0$ for any $\eta\in H^{2,0}(X)$. To see this, we note that
\begin{eqnarray}[ A]^*(\eta)= f_*( [\mathcal{Z}]\cup\Gamma^*\eta),\,\,\,\forall \,\, \eta\in H^{2,0}(X).
\end{eqnarray}
 Recalling that $\mathcal{Z}_x\in F^2{\rm CH}_0(\mathcal{Y}_x)$, we get that, at least after restriction to $\mathcal{Y}_U=f^{-1}(U)$ for some   Zariski dense open set
 $U$ of $X$,  $[\mathcal{Z}]_{\mid \mathcal{Y}_U}\in L^2H^4(\mathcal{Y}_U,\mathbb{Q})$, where $L^\bullet$ denotes the Leray filtration associated with the map $f$. On the other hand, as $\Gamma^*\eta_{\mid \mathcal{Y}_x}=0$, we get that
 $[\Gamma^*\eta]\in L^1H^2(\mathcal{Y}_U,\mathbb{C})$. It follows that
 $$ [\mathcal{Z}]\cup\Gamma^*\eta_{\mid \mathcal{Y}_U}\in L^3H^6(\mathcal{Y}_U,\mathbb{C})\subset {\rm Ker}\,(f_*:H^6(\mathcal{Y}_U,\mathbb{C})\rightarrow H^2(U,\mathbb{C})).$$
 Thus the class   $f_*( [\mathcal{Z}]\cup\Gamma^*\eta)$ vanishes in $H^2(U,\mathbb{C})$, hence vanishes in $H^{2,0}(X)$  since the composite map
 $$H^{2,0}(X)\rightarrow H^{2}(X,\mathbb{C})\rightarrow H^{2}(U,\mathbb{C})$$ is injective by coniveau (see \cite{deligne}).
We thus conclude from the claim and  (\ref{eqpourdiag}) that for any $\eta\in H^{2,0}(X)$,
\begin{eqnarray}
\label{eqpourdiagfinale}  N\eta  =[(Id_X,j)_*\mathcal{Z}']^*\eta\,\, {\rm in}\,\,H^{2,0}(X).
\end{eqnarray}
But $[(Id_X,j)_*\mathcal{Z}']^*\eta=[\mathcal{Z}']^*(j^*\eta)$, so (\ref{eqpourdiagfinale}) implies  that
$[\mathcal{Z}']^*: H^{2,0}(\Sigma)\rightarrow H^{2,0}(X)$ is surjective.
Hence $X$ satisfies the Lefschetz standard conjecture in degree $2$ by Proposition \ref{procharles}.
\end{proof}

We conclude this section with the following result that will be used in next section.
\begin{lemm} \label{lepropjoshi} Let $X,\,Y$ be  smooth projective varieties, and let $\Gamma\in {\rm CH}^2(X\times Y)$.
Then for any $z\in F^3_{BB}{\rm CH}_0(X)$, $\Gamma_*z=0$ in ${\rm CH}^2(Y)$.
\end{lemm}
\begin{proof} As $z\in F^3_{BB}{\rm CH}_0(X)$, the cycle $w:=\Gamma_*z\in {\rm CH}^2(Y)$ has the property that
for any surface $T\subset Y$, $w_{\mid T}=0$ in ${\rm CH}_0(T)$. One then applies the following
theorem of Joshi \cite{yoshi}.
\begin{theo} Let $Y$ be  a smooth projective variety and let $w\in {\rm CH}^2(Y)$ be a cycle such that
$w_{\mid T}=0$ for any surface $T\subset Y$. Then $w=0$.
\end{theo}
\end{proof}

\subsection{A  conjecture on $0$-cycles of hyper-K\"{a}hler manifolds  \label{secconjzero}}

We conclude with the formulation of a conjecture concerning the rational equivalence of points on a smooth projective
hyper-K\"ahler manifold $X$. Recall from Section \ref{seclef} the filtration
$$F^{3}_{BB}{\rm CH}_0(X)\subset F^{2}{\rm CH}_0(X):={\rm Ker}\,({\rm alb}_X:{\rm CH}_0(X)_{hom}\rightarrow {\rm Alb}(X)\otimes\mathbb{Q}).
$$
The rational equivalence of points on a hyper-K\"ahler manifold is a very intriguing subject, that started with work of Beauville
and the author \cite{beauvoi},  and  continued in \cite{huybrechtsspherical}, \cite{ogrady}, \cite{voiproc}, \cite{voisinremarks}, and more recently  in
the paper \cite{marian} which establishes the following remarkable result:
\begin{theo} \label{theomarian} Let $\mathcal{M}$ be a smooth projective moduli space of semistable sheaves on a projective  $K3$ surface $S$. Then
two points $[\mathcal{F}],\,[\mathcal{G}]\in \mathcal{M}$ have the same rational equivalence class in $\mathcal{M}$ if and only if
$c_2(\mathcal{F})=c_2(\mathcal{G})$ in ${\rm CH}_0(S)$.
\end{theo}
Of course, the condition is necessary, since, assuming for simplicity there is a universal sheaf
$\mathcal{F}_{\rm univ}$ on $\mathcal{M}\times S$, the Chern class
$$c_2(\mathcal{F}_{\rm univ})\in{\rm CH}^2(\mathcal{M}\times S)$$
induces a group morphism ${\rm CH}_0(\mathcal{M})\rightarrow  {\rm CH}_0(S)$ which maps the class of  $[\mathcal{F}]$ to
$c_2(\mathcal{F})$. The other direction is far from obvious and is a very intriguing statement. The following conjecture
is motivated by Theorem \ref{theomarian}.
\begin{conj}\label{conjsurpoints} Let $X$ be a smooth projective manifold whose algebra of holomorphic forms is
generated in degree $\leq 2$ (for example, a projective hyper-K\"ahler manifold). Then for $x,\,y\in X$,
$x=y$ in ${\rm CH}_0(X)$ if and only if $x=y$ in ${\rm CH}_0(X)/F^3_{BB}{\rm CH}_0(X)$.
\end{conj}
Theorem \ref{theomarian} establishes the conjecture when $X=\mathcal{M}$ is as above a smooth projective  moduli space of sheaves on a $K3$ surface. Indeed, if $[\mathcal{F}]=[\mathcal{G}]$ in
${\rm CH}_0(\mathcal{M})/F^3_{BB}{\rm CH}_0(\mathcal{M})$, then $c_2(\mathcal{F})=c_2(\mathcal{G})$ in ${\rm CH}_0(S)$ by the definition (\ref{F3BB}) of $F^3_{BB}$, and using  the existence of the  correspondence $c_2(\mathcal{F}_{{\rm univ}})$. Thus Theorem \ref{theomarian} gives $[\mathcal{F}]=[\mathcal{G}]$ in
${\rm CH}_0(\mathcal{M})$.

Let us  observe that  the Lefschetz standard conjecture in degree $2$, the generalized Hodge conjecture, and the nilpotence conjecture together  imply
Conjecture \ref{conjsurpoints}. Let us explain this fact assuming, as this is the case when $X$ is hyper-K\"ahler, that ${\rm dim}\,X=2n$ and   the algebra of holomorphic forms is generated in degree $2$, with  $H^{2,0}(X)=\mathbb{C}\sigma_X$. Then
$H^{i,0}(X)=0$ for $i$ odd and $H^{2i,0}(X)$ is generated by $\sigma_X^i$.  Denoting by $h_X\in {\rm NS}(X)_\mathbb{Q}\subset H^2(X,\mathbb{Q})$ a polarizing class on $X$, and assuming the Lefschetz standard conjecture in degree $2$ holds for $X$,
there exists a codimension $2$ cycle $\mathcal{Z}_{\rm lef}\in{\rm CH}^2(X\times X)$ such that
\begin{eqnarray}\label{eqactionsurformesdeg2}[\mathcal{Z}_{\rm lef}]^*(h_X^{2n-2}\sigma_X)=\sigma_X.
\end{eqnarray}
\begin{lemm}\label{leactionsurautreforme} One has for any $i\leq n$,
\begin{eqnarray}\label{eqactionsurformes}   [\mathcal{Z}_{\rm lef}^i]^*h_X^{2n-2i}\sigma_X^i=\lambda_i \sigma_X^i,
\end{eqnarray}
for some nonzero rational number $\lambda_i$.
\end{lemm}
\begin{proof}  We write
\begin{eqnarray}\label{eqkunneth} [\mathcal{Z}_{\rm lef}]=[\mathcal{Z}_{\rm lef}]^{2,2}+[\mathcal{Z}_{\rm lef}]^{4,0}+[\mathcal{Z}_{\rm lef}]^{0,4},
\end{eqnarray}
where $[\mathcal{Z}_{\rm lef}]^{i,j}\in H^i(X,\mathbb{Q})\otimes H^j(X,\mathbb{Q})$ denotes  the $(i,j)$ K\"unneth component of
$[\mathcal{Z}_{\rm lef}]$. Note that the K\"unneth components are algebraic in this case, since $[\mathcal{Z}_{\rm lef}]^{4,0}$, resp.
$[\mathcal{Z}_{\rm lef}]^{0,4}$, identifies to the restriction to $X\times x$, resp. $x\times X$, of $[\mathcal{Z}_{\rm lef}]$, where $x$ is any point of $X$. Furthermore, developing the
$i$-th power of (\ref{eqkunneth}), one sees immediately  that
$$[\mathcal{Z}_{\rm lef}^i]^*(h_X^{2n-2i}\sigma_X^i)=
 (([\mathcal{Z}_{\rm lef}]^{2,2})^i)^*(h_X^{2n-2i}\sigma_X^i).$$
In other words,  we can assume that $[\mathcal{Z}_{\rm lef}]=[\mathcal{Z}_{\rm lef}]^{2,2}$.
Over $\mathbb{C}$, we have the Hodge decomposition
$$H^2(X,\mathbb{C})=\mathbb{C}\sigma_X\oplus H^{1,1}(X)\oplus \overline{\mathbb{C}\sigma_X},$$
and the equation (\ref{eqactionsurformesdeg2}) then  says that
\begin{eqnarray}\label{eqkunnethHodge}[\mathcal{Z}_{\rm lef}]=\lambda ({\rm pr}_1^*\sigma_X\cup  {\rm pr}_2^*\overline {\sigma_X}+{\rm pr}_1^*\overline {\sigma_X}\cup  {\rm pr}_2^*\sigma_X)+R,
\end{eqnarray}
where the residual term $R$ belongs to the term $H^{1,1}(X)\otimes H^{1,1}(X)$ of the Hodge-K\"unneth decomposition of $H^4(X\times X,\mathbb{C})$.
When we take the $i$-th power of (\ref{eqkunnethHodge}), only the term
$\lambda^i ({\rm pr}_1^*\sigma^i\cup  {\rm pr}_2^*\overline {\sigma}^i)$ acts nontrivially on $h_X^{2n-2i}\sigma_X^i$ and it acts on $h_X^{2n-2i}\sigma_X^i$ by multiplication by the coefficient $ \lambda^i\int_Xh_X^{2n-2i}\sigma_X^i\overline{\sigma_X}^i$, which proves the result.
\end{proof}
\begin{coro}\label{corogamma} There exists a polynomial
$$P=\sum_{i=0}^n\mu_i \mathcal{Z}_{\rm lef}^i\cdot {\rm pr}_2^*h_x^{2n-2i}\in{\rm CH}^{2n}(X\times X),$$ such that
\begin{eqnarray}\label{eqpouractiondeGamma}
[P]^*\sigma_X^i=\sigma_X^i
\end{eqnarray}
for $i=0,\ldots, n$.
\end{coro}
\begin{proof} Let $\Gamma_i:=\mathcal{Z}_{\rm lef}^i\cdot {\rm pr}_2^*h_x^{2n-2i}$. We have
$$[\Gamma_i]^* \sigma_X^j=0\,\,{\rm for}\,\, j>i,$$
$$[\Gamma_i]^* \sigma_X^j=\lambda_{ij}\sigma_X^j\,\,{\rm for}\,\, j\leq i,$$
with $\lambda_j\not=0$ when $j=i$ by Lemma \ref{leactionsurautreforme}.
The matrix $M:=(\lambda_{ji})$ is thus invertible, and our condition on the column vector $^t(\mu_0,\ldots,\mu_n)$ is
$$M\begin{pmatrix}\mu_0\\ \vdots\\\mu_n\end{pmatrix}=\begin{pmatrix}1\\
\vdots\\1\end{pmatrix}.$$
\end{proof}
As the cycle
$ P=\sum_{i=0}^n \mathcal{Z}_{\rm lef}^i\cdot {\rm pr}_2^*h_X^{2n-2i} \in{\rm CH}^{2n}(X\times X)$
 acts as the identity on $H^{*,0}(X)$,
the cycle $\Delta_X-P\in{\rm CH}^{2n}(X\times X)$ acts as $0$ on $H^{*,0}(X)$, so that its image
$${\rm Im}\,[\Delta_X-P]^*: H^*(X,\mathbb{Q})\rightarrow H^*(X,\mathbb{Q})$$
has Hodge coniveau $\geq1$. Assuming the generalized Hodge conjecture, there should exist a divisor $D\subset X$ and a
cycle $\mathcal{Z}'\in {\rm CH}^{2n}(X\times X)$ supported on $D\times X$ such that
$$[\Delta_X-P]=[\mathcal{Z}']\,\,{\rm in}\,\,H^{4n}(X\times X,\mathbb{Q}).$$
Equivalently, the cycle $\Delta_X-P-\mathcal{Z}'$ is cohomologous to $0$, and the nilpotence conjecture
\cite{voe} predicts that a power
$(\Delta_X-P-\mathcal{Z}')^{\circ N}$ vanishes in ${\rm CH}^{2n}(X\times X)$, so in particular
the endomorphism

$$\phi:=(\Delta_X-P-\mathcal{Z}')_*:{\rm CH}_0(X)\rightarrow {\rm CH}_0(X)$$
is nilpotent.
We observe that, as $\mathcal{Z}'$ is supported on $D\times X$, $\mathcal{Z}'_*=0$ on ${\rm CH}_0(X)$, so
$$\phi=Id_{{\rm CH}_0(X)}-P_*.$$
Now assume $x,\,y\in X$ and  $x=y$ modulo $F^3_{BB}{\rm CH}_0(X)$. Then as already noticed,
$$\mathcal{Z}_{{\rm lef},x}=\mathcal{Z}_{{\rm lef},y}\,\,{\rm in}\,\,{\rm CH}^2(X),$$ hence
$\mathcal{Z}_{{\rm lef},x}^i=\mathcal{Z}_{{\rm lef},y}^i$ in ${\rm CH}^{2i}(X)$ for any $i$, so that $P_*x=P_*y$ in ${\rm CH}_0(X)$.
It thus follows that $\phi(x-y)=x-y$. But $\phi$ is nilpotent, so the last condition implies $x-y=0$, as we wanted.
\begin{rema}{\rm One may wonder why this proof does not work as well to show (assuming the general  conjectures) that for any   $0$-cycle $z=\sum_in_ix_i$ of degree $0$,
if $z=0$ modulo $F^3_{BB}{\rm CH}_0(X)$, then $z=0$, a statement that is clearly not expected
to be true since ${\rm CH}_0(X)/F^3_{BB}{\rm CH}_0(X)$ is expected to be a direct summand in the ${\rm CH}_0$ of a surface. The reason is that for a cycle $x-y$
$$P_*(x-y)=\sum_i \lambda_i h_X^{2n-2i}(\mathcal{Z}_{{\rm lef},x}^i-\mathcal{Z}_{{lef},y}^i)$$
is divisible by $\mathcal{Z}_{{lef},x}-\mathcal{Z}_{{lef},y}$, while
for general $0$-cycles of degree $0$, $P_*(z)$ is not a priori divisible by $(\mathcal{Z}_{\rm lef})_*(z)$.}
  \end{rema}
The same argument as above gives  the following statement linking the Lefschetz standard conjecture for degree $2$ cohomology  and Conjecture \ref{conjsurpoints}.
\begin{prop}\label{proppourconj}  Let $X$ be a  smooth projective  manifold with a codimension $2$ cycle $\mathcal{Z}_{\rm lef}\in{\rm CH}^2(X\times X)$ such that   there exists a  polynomial $P(y)=\sum_{i=0}^n\gamma_i y^i$ with coefficients in $\gamma_i\in {\rm CH}^{2n-2i}(X)$, such that for  any $x\in X$,
\begin{eqnarray} \label{eqpolpartoutpourpropetappli} x=P(\mathcal{Z}_x)\,\,{\rm in}\,\,{\rm CH}_0(X).
\end{eqnarray}
Then  $X$ satisfies Conjecture \ref{conjsurpoints}.
\end{prop}
\begin{proof} By Lemma \ref{lepropjoshi}, a $0$-cycle $z\in F^3_{BB}{\rm CH}_0(X)$ satisfies $\mathcal{Z}_*(z)=0$ in ${\rm CH}^2(X)$.
If $x=y$ modulo $ F^3_{BB}{\rm CH}_0(X)$, we thus get $\mathcal{Z}_x=\mathcal{Z}_y$ in  ${\rm CH}^2(X)$ and thus
(\ref{eqpolpartoutpourpropetappli}) implies that $x=y$ in ${\rm CH}_0(X)$.
\end{proof}
Proposition  \ref{proppourconj} applies when $X=F_1(Y)$ is the Fano variety of lines of a smooth cubic fourfold $Y$ (see \cite{bedo}). Indeed in this case   the following  quadratic formula, which is a particular case of (\ref{eqpolpartoutpourpropetappli})
\begin{eqnarray}
\label{eqformquadpourcub}  [l]=\alpha S_l^2+\Gamma. S_l+\gamma'\,\,{\rm in}\,\,{\rm CH}_0(X),
\end{eqnarray}
is proved in
\cite{voisinpamq},
for any line $l\subset Y$, where $S_l\subset F_1(Y)$ is the surface of lines meeting $l$.

This suggests that,   in the hyper-K\"ahler case, one can formulate a conjecture even stronger than Conjecture  \ref{conjsurpoints}
and we refer to \cite{shenvial} for a more precise formulation, namely that the cycle $\Delta_X-P-\mathcal{Z}'$ that appears in the argument above is
in fact $0$, and not only nilpotent.
\begin{conj} Let $X$ be a projective  hyper-K\"ahler manifold with polarization $h_X$. Then  any Lefschetz cycle  $\mathcal{Z}_{{\rm lef}}\in{\rm CH}^2(X\times X)$ for $h_X$ satisfies a polynomial
equation
\begin{eqnarray} \label{eqpolpartout} \Delta_X=P(\mathcal{Z}_{{\rm lef}})+W\,\,{\rm in}\,\,{\rm CH}^{2n}(X\times X),
\end{eqnarray}
where $P(\mathcal{Z}_{{\rm lef}})=\sum_i \mathcal{Z}_{{\rm lef}}^i{\rm pr}_2^*\gamma_i$ for some cycle $\gamma_i\in{\rm CH}^{2n-2i}(X)$, and
$W$ is supported on $D\times X$ for some divisor $D\subset X$.
\end{conj}

Turning to non hyper-K\"ahler projective manifolds, an easy  example of a variety satisfying Conjecture \ref{conjsurpoints} is a product of curves and surfaces. Indeed, if $X=\prod_{i=1}^m S_i$, where $S_i$ is either a curve or a surface, and $x=(x_1,\ldots, x_m)\in X$, $y=(y_1,\ldots, y_m)\in X$, then
$x=y$ in ${\rm CH}_0(X)/F^3_{BB}{\rm CH}_0(X)$ implies, using Definition (\ref{F3BB}), that $x_i=y_i$ in ${\rm CH}_0(S_i)$, so
that $x=y$ in ${\rm CH}_0(X)$. Another class of varieties satisfying the conjecture appears in  the following proposition.
\begin{prop}\label{propkummer} Let $X$ be a  desingularization of the quotient $K=A/\pm Id$ of an abelian variety by the $-1$-involution. Then $X$ satisfies Conjecture \ref{conjsurpoints}.
\end{prop}
\begin{proof} We first observe that ${\rm CH}_0(X)$ is isomorphic by pull-back to
the invariant part of ${\rm CH}_0(A)$ under the involution $-Id$, so we only have to
show that if $x,\,y\in A$ and
$\{x\}+\{-x\}=\{y\}+\{-y\}$ in ${\rm CH}_0(A)/F^3_{BB}{\rm CH}_0(A)$, then $\{x\}+\{-x\}=\{y\}+\{-y\}$ in ${\rm CH}_0(A)$.
Here we denote as usual  by $\{x\}$ the $0$-cycle corresponding to the point $x$, so as to distinguish the addition in $A$ and the addition in
${\rm CH}_0(A)$.
Let $\Theta$ be an ample divisor on $A$, determining an isogeny $A\rightarrow \widehat{A}$,
$$x\mapsto D_x=\Theta_x-\Theta.$$
Recalling the Pontryagin product
$*$ defined on $0$-cycles of $A$ by
$$z*z'=\mu_*(z\times z'),$$
where $\mu:A\times A\rightarrow A$ is the sum map,
 we have  Beauville's formulas in
\cite[Proposition 6]{beauvillefourier} which give  in particular  the following equalities, for any $x\in A$ and any integer $k$:
\begin{eqnarray}
\label{eqbeau}
  \frac{\theta^{g-k}}{(g-k)!}D_x^k=\frac{\theta^g}{g!}*\gamma(x)^{*k},
 \end{eqnarray}
where \begin{eqnarray}\label{eqgammadex}\gamma(x):=\{0_A\}-\{x\}+\frac{1}{2}(\{0_A\}-\{x\})^{*2}+\cdots+\frac{1}{g}(\{0_A\}-\{x\})^{*g}=-{\rm log }(\{x\})\,\,{\rm in}\,\, {\rm CH}_0(A).
\end{eqnarray} Here the logarithm is taken with respect to the
Pontryagin product $*$ and the expansion is finite because $0$-cycles of degree $0$ are nilpotent for the Pontryagin
product.  Note that $\{x\}*\{-x\}=\{0_A\}$, that is, $\{-x\}$ is the inverse of $\{x\}$ for the Pontryagin product, so that
$\gamma(x)=-\gamma(-x)$.

We can assume that $\frac{\theta^g}{g!}=d\{0_A\}$ for some nonzero integer $d$, so that the formula becomes
\begin{eqnarray}
\label{eqbeauavec0A}
  \frac{\theta^{g-k}}{(g-k)!}D_x^k=d\gamma(x)^{*k}
 \end{eqnarray}
 for any $k$.
We now add-up formulas (\ref{eqbeauavec0A}) for $x$ and $-x$. For odd $k$ we get $0$, using
$$D_x=-D_{-x},\,\,\,\gamma(x)=-\gamma(-x),$$ and for even $k=2r$ we get
\begin{eqnarray}
\label{eqbeausym}
 2\frac{\theta^{g-2r}}{(g-2r)!}D_x^{2r}=d\gamma(x)^{*2r}\,\, {\rm in}\,\, {\rm CH}_0(A).
 \end{eqnarray}
We now observe that if $\{x\}+\{-x\}=\{y\}+\{-y\}$ in ${\rm CH}_0(A)/F^3_{BB}{\rm CH}_0(A)$, then, using Lemma  \ref{lepropjoshi},
$ D_x^2= D_y^2$  in ${\rm CH}^2(A)$, since clearly
$x\mapsto D_x^2$ is both $-1$-invariant and induced by a codimension $2$ self-correspondence of $A$.
So we conclude from (\ref{eqbeausym}) that, for any $r$
\begin{eqnarray}
\label{eqfinfinfin0}\gamma(x)^{*2r}=\gamma(y)^{*2r}\,\, {\rm in}\,\, {\rm CH}_0(A).\end{eqnarray}
Finally, we have (using again nilpotence so that the formal series reduce  in fact to finite sums).
$$\{x\}={\rm exp}(\gamma(x)),\,\,\{-x\}={\rm exp}(-\gamma(x))\,\,{\rm in}\,\,{\rm CH}_0(A)$$
and similarly for $y$, so that
\begin{eqnarray}
\label{eqfinfinfin} \{x\}+\{-x\}={\rm exp}(\gamma(x))+{\rm exp}(-\gamma(x))=2\sum_{r\geq 0}\frac{\gamma(x)^{*2r}}{(2r)!}\,\,{\rm in}\,\,{\rm CH}_0(A),
\end{eqnarray}
and similarly for $y$. Equations  (\ref{eqfinfinfin}) and (\ref{eqfinfinfin0}) imply $\{x\}+\{-x\}=\{y\}+\{-y\}$ in ${\rm CH}_0(A)$.
\end{proof}

We conclude with the following analogue  of Theorem \ref{theomarian} in the case of punctual  Hilbert schemes.
\begin{theo} \label{theolecasdehilb}Let $\Sigma$ be  a smooth surface. Then for any $n$, and any $z,\,z'\in \Sigma^{[n]}$,
$z$ is rationally equivalent to $z'$ in $\Sigma^{[n]}$ if and only if the corresponding $0$-cycles
 $Z$, $Z'$ of $\Sigma$ are rationally equivalent in $\Sigma$. A fortiori $\Sigma^{[n]}$ satisfies  Conjecture \ref{conjsurpoints} since a $0$-cycle in $F^3_{BB}{\rm CH}_0(\Sigma^{[n]})$ is annihilated
by $I_*$, where $I\subset \Sigma^{[n]})\times \Sigma$ is  the incidence correspondence.
\end{theo}

\begin{proof} Note that ${\rm CH}_0(\Sigma^{[n]})={\rm CH}_0(\Sigma^{(n)})$ for any $n$, so we will work with $\Sigma^{(n)}$. (It is  a standard fact that the quotient singularities of $\Sigma^{(n)}$ allow to do intersection theory at least with $\mathbb{Q}$-coefficients. The computations we make should be thought as computations in the invariant part under $\mathfrak{S}_n$ of the Chow groups of $\Sigma^n$.)
We start with the following lemma.
\begin{lemm} \label{newlemmapourhilb} For any integer $k\geq 0$ and any  $z\in \Sigma^{(k)}$, the addition
$$\mu_z:\Sigma^{(n)}\rightarrow \Sigma^{(k+n)},$$
$$w\mapsto w+z,$$
induces an injective map
$\mu_{z\,*}:{\rm CH}_0(\Sigma^{(n)})\rightarrow {\rm CH}_0(\Sigma^{(n+k)})$.
\end{lemm}
\begin{proof} Reasoning by induction on $k$, it clearly suffices to prove the result when $k=1$, so $z$ consists of one point $z\in\Sigma$.
The incidence  correspondence
$$\Gamma_{n+1}\in {\rm CH}^{2n}( \Sigma^{(n+1)}\times \Sigma^{(n)})$$
given by the nested symmetric product (one could work with the Hilbert schemes at this point)
induces a morphism
$$\Gamma_{n+1\,*}:{\rm CH}_0(\Sigma^{(n+1)})\rightarrow {\rm CH}_0(\Sigma^{(n)}).$$
One clearly  has
\begin{eqnarray}\label{eqnafindefin} \Gamma_{n+1\,*}\circ \mu_{z\,*} =Id_{{\rm CH}_0(\Sigma^{(n)})}+\mu_{z\,*}\circ \Gamma_{n\,*}.
\end{eqnarray}
A cycle in ${\rm Ker}\, \mu_{z\,*}$ is thus in the image of $\mu_{z\,*}$. Using the equations (\ref{eqnafindefin}) for
$n-1,\,n-2,\ldots1$, one proves similarly that
it is in the image of $\mu_{2z\,*},\, \mu_{3z\,*}$ and finally that it is zero.
\end{proof}
Let now $z,\,z'\in \Sigma^{[n]}$ be two points such that the corresponding $0$-cycles $Z,\,Z'$ are rationally equivalent in $\Sigma$.
Then there exist an effective $0$-cycle $W\in \Sigma^{(k)}$ and a rational curve in $\Sigma^{(n+k)}$ passing through
the two points $\mu_w(z)$ and $\mu_w(z')$. Thus $\mu_{w\,*}(z)=\mu_{w\,*}(z')$ in ${\rm CH}_0(\Sigma^{(n+k)})$. By Lemma
\ref{newlemmapourhilb}, $z=z'$ in ${\rm CH}_0(\Sigma^{(n)})$.
\end{proof}
We can also prove Theorem \ref{theolecasdehilb} by establishing a polynomial formula of the form
(\ref{eqpolpartout}) in this case and applying Proposition \ref{proppourconj}. Concretely, for  the symmetric product $X=\Sigma^{(n)}$, consider the codimension $2$ cycle
$\mathcal{Z}_{\rm lef}\in{\rm CH}^2(X\times X)$ defined as the incidence
correspondence:
$$\mathcal{Z}_{\rm lef}=\{(z_1,z_2)\in \Sigma^{(n)}\times \Sigma^{(n)},\,Z_1\cap Z_2\not=\emptyset\}.$$
\begin{prop}\label{propoursymfin}  There exist cycles $\gamma_i\in {\rm CH}^{2n-2i}(X)$ such that
the cycle $P=\sum_i \mathcal{Z}_{\rm lef}^i{\rm pr}_2^*\gamma_i$ satisfies, for any $z\in X$,
\begin{eqnarray}\label{eqpourpolhilbder} z=P(\mathcal{Z}_{{\rm lef},z})\,\,{\rm in}\,\,{\rm CH}_0(X).
\end{eqnarray}

\end{prop}

\begin{proof} It suffices to check  equality (\ref{eqpourpolhilbder})  after pull-back to $\Sigma^n$.  Let $z=x_1+\ldots+z_n\in \Sigma^{(n)}$ and denote by $\tilde{z}\in {\rm CH}_0(\Sigma^n)$ the pull-back of $\{z\}\in {\rm CH}_0(\Sigma^{(n)})$. Then $\mathcal{Z}_{{\rm lef},z}=\sum_i x_i+\Sigma^{(n-1)}$. In order to compute the self-intersection of this cycle, we pull-back to $\Sigma^n$, and denote the resulting cycle by $\mathcal{Z}_{{\rm lef},z}'$. We get by definition, assuming the points $x_i$ are distinct
$$\mathcal{Z}_{{\rm lef},z}'=\sum_{i,j} {\rm pr}_j^*x_{i}.$$
Using the fact that cycles $Z_{ij}:={\rm pr}_j^*x_{i}$ satisfy $Z_{ij}Z_{kj}=0$ for any $j,\,k$,
we get
$$ (\mathcal{Z}_{{\rm lef},z}')^n =\sum_{ f} Z_f,$$
where  $f$ runs through the set of applications from $\{1,\ldots,n\}$ to itself and
\begin{eqnarray}
\label{eqdefZsigmaf} Z_{f}=\prod_{i=1}^n Z_{f(i) i}=(x_{f(1)},\ldots, x_{f(n)})\in{\rm CH}_0(\Sigma^n).
\end{eqnarray}
We now stratify the set of maps $f$ according to their combinatorics. Up to the action of the symmetric group $\mathfrak{S}_n$, such a map is characterized by the image of $f$ and the partition of $\{1,\ldots,n\}$ given by the non-empty preimages $f^{-1}(s)\subset \{1,\ldots,n\}$.
The partition $(1,\ldots,1)$ where all sets have cardinality $1$ produce
the term $\sum_{\sigma\in \mathfrak{S}_n}(x_{\sigma(1)},\ldots,x_{\sigma(n)})$ which is the pull-back  $\tilde{z}\in{\rm CH}_0(\Sigma^n)$ of
the point $z=x_1+\ldots+x_n\in \Sigma^{(n)}$.
The next case is the case of a partition $(2,1,\ldots,1)$ for which exactly two points have the same image, the map $f$ being injective on the remaining set. The contribution
of these partitions to (\ref{eqdefZsigmaf}) is the sum
\begin{eqnarray}\label{eqcasdemulti2} \sum_{i\not=j,\sigma\in \mathfrak{S}_n}\sigma(x_1,\ldots, x_i,x_i,x_{i+1},\ldots, \widehat{x_{j+1}},\ldots,x_n),
\end{eqnarray}
where in each $n$-uple, exactly one  point appears twice and another point is missing.
We observe
that (\ref{eqcasdemulti2})  appears in the development of
\begin{eqnarray}\label{eqcestlader} (\mathcal{Z}_{{\rm lef},z}')^{n-1}\cdot \Delta_{(2,1,\ldots,1)},
\end{eqnarray}
 where the codimension $2$ cycle $\Delta_{(2,1,\ldots,1)}$ is the sum of the partial diagonals
$\Delta_{ij}$  where $y_i=y_j$.
Developing (\ref{eqcestlader}), we   get that
$$\tilde{z}= (\mathcal{Z}_{{\rm lef},z}')^n-(\mathcal{Z}_{{\rm lef},z}')^{n-1}\cdot \Delta_{(2,1,\ldots,1)}+R_z\,\,{\rm in}\,\,{\rm CH}_0(\Sigma^n),$$
where the cycle $R_z$ consists in a sum of terms $(x_{f(1)},\ldots, x_{f(n)})$ for which the image of the  map $f$ has cardinality
$\leq n-2$. More generally, one checks by induction on the cardinality of $f$ that Proposition \ref{propoursymfin} holds, where one can take for the cycles $\gamma_i$ a combination with rational coefficients of the diagonals
$\Delta_I$ where $|I|=n-i$. Here the cardinality of $I=\{I_1,\ldots, I_k\}$ with $\{1,\ldots,n\} =I_1\sqcup \ldots\sqcup I_k$ is the number
$k$ and the diagonal $\Delta_I\subset \Sigma^n$ is defined by the equations $y_j=y_k$ when $i,\,k\in I_l$ for some $l$.
\end{proof}

CNRS, Institut de Math\'{e}matiques de Jussieu-Paris rive gauche

claire.voisin@imj-prg.fr
    \end{document}